\documentclass[11pt]{amsart}
\usepackage[pdftex]{graphicx}
\usepackage{amssymb}
\usepackage{hyperref}
\usepackage{amsmath}
\usepackage{amsfonts}
\usepackage{amsthm}
\usepackage{amssymb}
\usepackage{mathrsfs}
\usepackage{enumitem}
\usepackage{xcolor}
\newtheorem{theorem}{Theorem}[section]
\newtheorem{proposition}[theorem]{Proposition}
\newtheorem{corollary}[theorem]{Corollary}

\newtheorem{lemma}[theorem]{Lemma}

\theoremstyle{definition}
\newtheorem{definition}[theorem]{Definition}
\theoremstyle{remark}
\newtheorem{remark}[theorem]{Remark}

\newtheorem{example}[theorem]{Example}

\renewcommand{\leq}{\leqslant}
\renewcommand{\geq}{\geqslant}

\newcommand{\bd}{\partial}

\newcommand{\mE}{\mathcal E}

\newcommand{\mL}{\mathcal L}

\newcommand{\mN}{\mathcal N}
\newcommand{\la}{\langle}
\newcommand{\ra}{\rangle}
\newcommand{\diam}{\mathrm{diam}}
\newcommand{\dist}{\mathrm{dist}}

\newcommand{\supp}{\mathrm{supp}}
\newcommand{\eps}{\varepsilon}

\newcommand{\Diff}{\mathrm{Diff}}

\newcommand{\inj}{\mathrm{inj}}

\newcommand{\optV}{\mathcal V}
\newcommand{\spaceV}{\mathcal P}
\newcommand{\Sing}{\mathrm{Sing}}
\newcommand{\II}{\mathrm{II}}
\newcommand{\Hess}{\mathrm{Hess}}
\newcommand{\Ric}{\mathrm{Ric}}

\DeclareMathOperator{\area}{\mathrm{Area}}
\DeclareMathOperator{\vol}{\mathrm{Vol}}
\DeclareMathOperator{\ind}{\mathrm{ind}}
\DeclareMathOperator{\nul}{\mathrm{nul}}
\numberwithin{equation}{section}

\begin{document}

\title[Harmonic maps in higher dimensions]{Existence of harmonic maps and eigenvalue optimization in higher dimensions}

\author[M. Karpukhin]{Mikhail Karpukhin}
\address[Mikhail Karpukhin]{Mathematics 253-37, California Institute of Technology, 
Pasadena, CA 91125, USA}
\email{mikhailk@caltech.edu}
\author[D. Stern]{Daniel Stern}
\address[Daniel Stern]{Department of Mathematics, University of Chicago,
5734 S University Ave
Chicago IL, 60637, USA}
\email{dstern@uchicago.edu}
\date{}

\begin{abstract} 

We prove the existence of nonconstant harmonic maps of optimal regularity from an arbitrary closed manifold $(M^n,g)$ of dimension $n>2$ to any closed, non-aspherical manifold $N$ containing no stable minimal two-spheres. In particular, this gives the first general existence result for harmonic maps from higher-dimensional manifolds to a large class of positively curved targets. In the special case of the round spheres $N=\mathbb{S}^k$, $k\geq 3$, we obtain a distinguished family of nonconstant harmonic maps $M\to \mathbb{S}^k$ of index at most $k+1$, with singular set of codimension at least $7$ for $k$ sufficiently large. Furthermore, if $3\leq n\leq 5$, we show that these smooth harmonic maps stabilize as $k$ becomes large, and correspond to the solutions of an eigenvalue optimization problem on $M$, generalizing the conformal maximization of the first Laplace eigenvalue on surfaces.


\end{abstract}
\maketitle

\section{Introduction}

\subsection{Existence of harmonic maps}
A map $u\colon M\to N$ between Riemannian manifolds $(M,g)$ and $(N,h)$ is said to be \emph{harmonic} if it is a critical point for the Dirichlet energy
 $$
 E(u)=\frac{1}{2}\int_M|du|^2_{g,h}\,dv_g
 $$ 
 on the space of maps from $M$ to $N$. Generalizing classical questions about the existence of closed geodesics in a given Riemannian manifold, existence and regularity problems for harmonic maps between higher-dimensional manifolds have played an important role in the development of geometric analysis over the past sixty years (see, e.g. \cite[Section 3]{YauSurv}).

One natural strategy for producing harmonic maps is to minimize energy in a homotopy class of maps from $M$ to $N$. For targets $N$ of nonpositive sectional curvature, this was carried out successfully by Eells and Sampson in the landmark paper \cite{EeSa64}. Indeed, combined with subsequent work of Hartman \cite{Hartman}, the results of \cite{EeSa64} give a complete picture of the space of harmonic maps from a general closed manifold $(M^n,g)$ to a closed target $(N^k,h)$ of nonpositive sectional curvature: every homotopy class contains a nonempty, connected set of smooth harmonic representatives, all of which minimize energy in the class. Moreover, if $N$ has strictly negative sectional curvature, then the harmonic representative of each homotopy class is unique.

For general target manifolds $N$, the situation is quite different. For general higher-dimensional domains $M^n$ and targets $N^k$, there often exist nontrivial homotopy classes of maps which admit no energy-minimizing representative even in a weak sense, as distinct path components of $C^1(M,N)$ may merge in the weak or strong topologies on the space $W^{1,2}(M,N)$ of finite-energy maps between the manifolds. On the other hand, as in the study of closed geodesics, in many cases one can produce interesting non-energy-minimizing harmonic maps $M\to N$ via Morse-theoretic or min-max methods. In the influential paper \cite{SaUhl77}, Sacks and Uhlenbeck developed a Morse-theoretic approach to the study of harmonic maps from the $2$-sphere to general targets, obtaining existence results that later saw elegant applications to other classical problems in differential geometry, notably in the work of Siu--Yau \cite{SiuYau} and Micallef--Moore \cite{MM88}. 

While the work of Sacks-Uhlenbeck led to major improvements in the existence and compactness theory for harmonic maps from surfaces, the space of harmonic maps from manifolds of dimension $n\geq 3$ into general targets remains rather poorly understood, due in part to the presence of singularities in harmonic maps arising from variational methods on higher-dimensional domains (see, e.g., \cite{SU82}), and--perhaps more seriously--the more complicated higher-dimensional counterpart of bubbling for families of harmonic maps and approximations thereof (see \cite{Lin99}). In particular, compactness for harmonic maps and Palais-Smale sequences for the energy functional fails more dramatically in dimension $\geq 3$, making it difficult to implement min-max constructions in general.

The first result of the present paper establishes existence of nontrivial harmonic maps from arbitrary closed manifolds $M^n$ of dimension $n\geq 3$ into a large class of targets $N^k$, via min-max methods. In what follows, recall that a map $u\in W^{1,2}(M,N)$ is said to be a stationary harmonic map if it is a critical point for the energy functional in the following strong sense: in addition to solving the weak Euler-Lagrange equations for $E(u)$, $u$ satisfies $\left.\frac{d}{dt}\right|_{t=0}E(u_t)=0$ for inner variations $u_t=u\circ \Phi_t$, where $t\mapsto\Phi_t$ is a smooth family of diffeormorphisms of the domain $M$, see Section~\ref{sec:minmax} for details.

\begin{theorem}\label{exthm} Let $(N^k,h)$ be a closed Riemannian manifold containing no stable minimal two-spheres, and suppose $\pi_{\ell}(N)\neq 0$ for some $\ell \geq 3$. Then for any closed manifold $(M^n,g)$ of dimension $n\geq 3$, there exists a nonconstant stationary harmonic map
$$
u\colon (M^n,g)\to (N^k,h)
$$
of Morse index $\ind_E(u)\leq \ell+1$ as a critical point of $E$, smooth away from a set $\Sigma\subset M^n$ of dimension $\dim(\Sigma)\leq n-3$.
Moreover, if for some $m>2$, $N^k$ admits no nonconstant stable $0$-homogeneous harmonic map $v\colon \mathbb{R}^m\to N^k$, then the singular set $\Sigma=\Sing(u)$ has dimension $\dim(\Sigma)\leq n-m-1$.
\end{theorem}

Here, the condition that $N$ admits no stable minimal two-spheres is meant in the sense of branched minimal immersions; equivalently (cf. \cite{EjiriMicallef}), there exists no nonconstant stable harmonic map $\mathbb{S}^2\to N$. In \cite{Lin99}, Lin showed that stationary harmonic maps into targets $N$ carrying \emph{no} minimal two-spheres satisfy strong compactness and partial regularity properties; however, it follows from the results of \cite{SaUhl77, MM88} that such targets must be aspherical, and therefore unlikely to support a wealth of variational constructions beyond minimization with respect to a prescribed action $\pi_1(M)\to \pi_1(N)$ on fundamental groups. By contrast, the only obvious topological constraint following from the assumption that $N$ admits no \emph{stable} minimal two-spheres is the vanishing of the second homotopy group $\pi_2(N)$.

The significance of the condition that a target $N$ admit no stable harmonic maps from $\mathbb{S}^2$ was first noticed by Hsu \cite{Hsu05}, who observed that stable stationary harmonic maps to such targets enjoy strong compactness properties and, as a byproduct, optimal partial regularity results analogous to those obtained by Schoen--Uhlenbeck for energy-minimizing maps \cite{SU82}. In the special case of maps to $\mathbb{S}^k$ with $k\geq 3$, similar observations were made by Hong-Wang \cite{HW99} and Lin-Wang \cite{LW06}. Theorem \ref{exthm} rests largely on the observation that, under the same assumptions on the target manifold $N$, the same compactness and partial regularity results hold for the space of stationary harmonic maps $M\to N$ satisfying a uniform Morse index bound, and similarly for maps critical for suitable relaxations of the Dirichlet energy--namely, the Ginzburg--Landau-type energies considered in \cite{CS} and \cite{LW99}. With these analytic ingredients in place, the harmonic maps of Theorem \ref{exthm} are obtained from a min-max construction generalizing those studied in \cite{SternGL, Sternpharm, Riv19, Riv20, KS1}. 

For examples of targets satisfying the hypotheses of Theorem \ref{exthm}, note that the results of \cite{MM88} show that there are no stable minimal two-spheres in a manifold $N^k$ of dimension $k\geq 4$ with positive isotropic curvature. The same is true for $3$-manifolds $N^3$ of positive Ricci curvature: it is well-known that such manifolds admit no stable immersed minimal surfaces with trivial normal bundle \cite{Sim68}, and consequently no stable minimal immersions $\mathbb{S}^2\to N$, since any such immersion can be lifted to one $\mathbb{S}^2\to \tilde{N}$ with trivial normal bundle on a double cover of $N$ if $N$ is nonorientable. That the same holds for branched minimal immersions can be seen by a straightforward application of the log cutoff trick near branch points, cf. \cite[Section 3]{M85} for related observations. In particular, while the existence theory for harmonic maps into negatively curved targets has been well understood since the 1960s \cite{EeSa64}, Theorem \ref{exthm} and its proof show that the existence theory for harmonic maps into targets satisfying certain curvature positivity conditions can be fruitfully explored via Morse-theoretic methods.

\begin{remark} The simplest example of a target failing to satisfy the hypotheses of Theorem \ref{exthm} is the standard $2$-sphere $N=\mathbb{S}^2$, for which any holomorphic or anti-holomorphic self-map $\mathbb{S}^2\to \mathbb{S}^2$ is stable. In particular, we note that the methods of the present paper cannot be used in a direct way to advance the min-max theory for harmonic maps $\mathbb{S}^3\to \mathbb{S}^2$ considered by Rivi\`ere in \cite{Riv19, Riv20}. More generally, the complex projective spaces $N=\mathbb{CP}^m$ represent an interesting borderline case where the compactness results on which Theorem \ref{exthm} relies fail, though the space of stable minimal 2-spheres in $N$ is well-understood (see, e.g. \cite{SiuYau}).
\end{remark}

\subsection{Harmonic maps to spheres}
\label{sec:intro_spheres}
Of particular interest to us is the case of maps to the standard spheres $N=\mathbb{S}^k$ of dimension $k\geq 3$. In particular, taking $N=\mathbb{S}^k$ and $\ell=k$ in Theorem \ref{exthm} and appealing to the regularity results of \cite{LW06} for stable stationary harmonic maps to spheres, we obtain the following result, giving existence of a canonical family of sphere-valued harmonic maps from every closed Riemannian manifold.

\begin{corollary}
\label{sphere.cor}
For any closed manifold $(M^n,g)$ of dimension $n\geq 3$ and any $k\geq 3$, there is a nonconstant stationary harmonic map
$$
u_k\colon M^n\to \mathbb{S}^k
$$
of Morse index $\ind_E(u_k)\leq k+1$, smooth away from a closed set $\Sigma\subset M$ of dimension
$$\dim(\Sigma)\leq n-k-1\text{\hspace{2mm} if\hspace{2mm} }3\leq k\leq 5,$$
$$\dim(\Sigma)\leq n-6\text{\hspace{2mm} if\hspace{2mm} }6\leq k\leq 9,$$
or
$$\dim(\Sigma)\leq n-7\text{\hspace{2mm} if\hspace{2mm} }k\geq 10.$$
In particular, if $3\leq n\leq 5$, then $u_k\colon M^n\to \mathbb{S}^k$ is smooth for all $k\geq n$. 
\end{corollary}

A natural question to ask is how these maps $u_k$ depend on the dimension $k$ of the target sphere. The following proposition provides a partial answer, which plays an important role in applications.

\begin{proposition}
\label{prop:intro_stab}
For any closed manifold $(M^n,g)$ of dimension $3\leq n\leq 5$, there exists $k_0(M,g)\in \mathbb{N}$ such that for any $k\geq k_0$ and any harmonic map $u\colon M\to\mathbb{S}^k$ with 
$$\ind_E(u)\leq k+1,$$
there exists a totally geodesic subsphere $\mathbb{S}^{k_0}\subset \mathbb{S}^k$ of dimension $k_0$ such that $u(M)\subset \mathbb{S}^{k_0}$.
\end{proposition}

\subsection{Applications to eigenvalue optimization}
\subsubsection{Motivation and overview}
Let $(\Sigma,g)$ be a closed Riemannian surface and denote by 
$$0=\lambda_0(\Sigma,g)<\lambda_1(\Sigma,g)<\ldots$$
its (positive) Laplacian eigenvalues. Fixing a conformal class $[g]=\{e^{2\omega}g,\,\omega\in C^\infty(\Sigma)\}$, consider the normalized eigenvalue functionals 
$$
\bar\lambda_m(\Sigma,g):=\lambda_m(\Sigma,g)\area(\Sigma,g).
$$
 According to the 
 foundational results of Nadirashvili~\cite{NadirashviliT2} and El Soufi-Ilias \cite{ESIextremal}, metrics critical for $\bar{\lambda}_m(\Sigma,g)$ within the conformal class $[g]$ correspond to harmonic maps from $(\Sigma,[g])$ to spheres. In recent decades, this observation has proved to be a crucial tool in the study of the suprema
$$
\Lambda_m(\Sigma,[g]) = \sup_{h\in[g]}\bar\lambda_m(\Sigma,h),
$$
see e.g.~\cite{KS1, KNPP1, KNPP2, KRP2, Petrides, Petrides2}. In particular, $2$-dimensional counterparts of Corollary~\ref{sphere.cor} and Proposition~\ref{prop:intro_stab} were used in~\cite{KS1} to establish existence and regularity of metrics achieving $\Lambda_1(\Sigma,[g])$. These results have since found application to the study of the Steklov maximization problem on surfaces with boundary \cite{KS1,KS2}, and stability phenomena for the $\bar{\lambda}_1$-maximization problem \cite{KNPS}. 
It is then natural to ask whether Corollary~\ref{sphere.cor} and Proposition~\ref{prop:intro_stab} themselves can be used to study some eigenvalue optimization problem on higher-dimensional manifolds. Below we describe such a problem.

A natural way to generalize $\Lambda_m(\Sigma,[g])$ to higher-dimensional manifolds $(M^n,g)$ is to once again maximize appropriately normalized eigenvalues in a fixed conformal class, studying the suprema
$$
\Lambda_m(M^n,[g]):= \sup_{h\in[g]}\bar\lambda_m(M^n,h):=\sup_{h\in[g]}\lambda_m(M^n,g)\vol(M^n,g)^\frac{2}{n}.
$$
However, as is shown in~\cite{KM}, critical metrics for this problem correspond to {\em $n$-harmonic maps}, the conformally invariant analog of classical harmonic maps, which suggests that applications of classical harmonic maps lie elsewhere. The key lies in reformulating the conformal maximization problem on surfaces. Given a smooth positive function $\beta>0$ consider the eigenvalues 
$$0=\lambda_0(\Sigma,g,\beta)<\lambda_1(\Sigma,g,\beta)<\ldots$$ 
associated to the problem
$$
\Delta_g f = \lambda \beta f,
$$
i.e., the eigenvalues of the weighted Laplacian $\beta^{-1}\Delta_g$. 
Conformal covariance of the Laplacian on surfaces implies that if $h=e^{2\omega}g\in [g]$ and $f$ is a $\Delta_h$-eigenfunction, then
$$
\Delta_g f = e^{2\omega}\Delta_{h}f = \lambda e^{2\omega}f,
$$
so that $\lambda_m(\Sigma, h) = \lambda_m(\Sigma, g, e^{2\omega})$ for all $m$. Therefore, one has that for any surface $\Sigma$,
\begin{equation}
\label{eq:intro_equiv}
\Lambda_m(\Sigma,[g]) = \sup_{0<\beta\in C^\infty(\Sigma)}\left\{ \lambda_m(\Sigma,g,\beta)\int_\Sigma \beta\,dv_g\right\}.
\end{equation}
Moreover, it follows from~\cite[Proposition 5.1]{GKL} that one can relax the condition in the r.h.s to $0\leq \beta$, $\beta\not\equiv 0$.

On manifolds $M^n$ of dimension $n\geq 3$, the Laplacian is no longer conformally covariant and equality~\eqref{eq:intro_equiv} fails. Thus, we introduce the quantities
\begin{equation}
\label{def:optV}
\optV_m(M^n,g) = \sup_{0\leq\beta\in C^\infty(M^n)}\left\{ \lambda_m(M^n,g,\beta)\int_{M^n} \beta\,dv_g\right\},
\end{equation}
which generally differ from $\Lambda_m(M^n,[g])$.

Our first observation is  that the quantities  $\optV_{m}(M^n,g)$ in many ways behave similarly to $\Lambda_m(\Sigma,[g])$. For example, by a well-known result of Korevaar~\cite{Korevaar} one has $\Lambda_m(\Sigma,[g])\leq Cm$ on a surface $\Sigma$. The same method of proof gives $\optV_m(M^n,g)\leq Cm^{\frac{2}{n}}$, see~\cite[Theorem 5.4]{GNY}. Another example is the classical Hersch's inequality~\cite{Hersch} stating that the round metric $g_{\mathbb{S}^2}$ achieves $\Lambda_1(\mathbb{S}^2,[g_{\mathbb{S}^2}])$. We show below that essentially the same proof can be used to show that $\optV_1(\mathbb{S}^n,g_{\mathbb{S}^n})$ is achieved by a constant density function. Furthermore, the Laplacian with density appears naturally as a homogenization limit of the Steklov problem, see~\cite{GL}. This feature is independent of dimension and has been used in~\cite{GL, GKL, KS2} to relate optimization problems for Steklov and Laplace eigenvalues on surfaces. Below we outline generalizations of these results to $\optV_m(M^n,g)$. Finally, and most importantly to us, we prove in Proposition~\ref{prop:critical1} that densities critical for $\lambda_m(\beta)\|\beta\|_{L^1}$ are precisely the energy densities of classical sphere-valued harmonic maps $(M,g)\to\mathbb{S}^k$. This observation 
provides the bridge between the results of Section~\ref{sec:intro_spheres} and the study of the quantities $\optV_m(M^n,g)$. 

\subsubsection{Existence and regularity of maximizers for $\optV_1(M,g)$.}
On a closed surface $(\Sigma,g)$, it is shown in~\cite{KS1} that, for the harmonic maps $u_k\colon (\Sigma,g)\to \mathbb{S}^k$ given by the two-dimensional analog of Corollary~\ref{sphere.cor}, the corresponding conformal metrics $g_k=\frac{1}{2}|du_k|_g^2g$ maximize $\bar\lambda_1(\Sigma,g)$ in $[g]$ provided $k$ is sufficiently large, i.e. $\Lambda_1(\Sigma,[g]) = \bar\lambda_1(\Sigma,g_k)$. The following theorem is a direct generalization of these results to higher dimensions $\leq 5$.

\begin{theorem}\label{stab.etc} On any closed Riemannian manifold $(M^n,g)$ of dimension $3\leq n\leq 5$, for $k\geq k(M,g)$ sufficiently large, the energy densities $e(u_k)=|du_k|^2_g$ of the 
 smooth harmonic maps $u_k$ constructed in Corollary \ref{sphere.cor} realize $\mathcal{V}_1(M,g)$. Namely, the components of $u_k$ are $\lambda_1(M^n,g,e(u_k))$-eigenfunctions, $\lambda_1(M^n,g,e(u_k)) = 1$ and
$$\mathcal{V}_1(M,g)=2E(u_k).$$
\end{theorem}

In addition to establishing the existence of smooth densities realizing $\mathcal{V}_1(M)$, our methods also imply that the maximization problem is unchanged if one considers non-negative densities $\beta\in L^{\frac{n}{2}}(M)$, or even certain ``admissible" measures $\mu$ for which the map $W^{1,2}(M,g)\to L^2(M,\mu)$ is compact. Furthermore, we are able to generalize the regularity result of~\cite{KS1} by showing that any maximal admissible measure has to be smooth, see Theorem~\ref{ev:main_theorem} for details. 


\subsubsection{Examples and relation to $\Lambda_1(M^n,[g])$.}
The following theorem provides a basic tool for computing $\optV_1(M^n,g)$ for some 
special manifolds $(M^n,g)$.

\begin{theorem}
\label{thm:ev_exmp}
Let $(M^n,g)$ be a closed Riemannian manifold, $n\geq 3$. Assume that there exists a minimal immersion $u\colon M^n\to\mathbb{S}^k$ such that the induced metric $g_u=u^*g_{\mathbb{S}^k}\in [g]$. Then for any density $0\leq\beta\in C^\infty(M^n)$ one has
$$
\lambda_1(M^n,g,\beta)\int_{M^n}\beta\,dv_g\leq n\left(\vol(M,g_u)\right)^{\frac{2}{n}}\vol(M,g)^\frac{n-2}{n}.
$$
If $(M^n,[g])\not = (\mathbb{S}^k, [g_{\mathbb{S}^k}])$ or if
 $u$ is not a conformal automorphism of $\mathbb{S}^k$, 
then equality occurs iff the components of $u$ are $\lambda_1(M^n,g_u)$-eigenfunctions, $g=a g_u$ for some constant $a\in (0,\infty)$, and $\beta\equiv b\in(0,\infty)$ is constant. If $(M^n,[g])=(\mathbb{S}^k,[g_{\mathbb{S}^k}])$ and $u$ is a conformal automorphism, then equality holds iff $g=a g_{u'}$ for some $a\in (0,\infty)$ and a (possibly different) conformal automorphism $u'$ of $\mathbb{S}^k$, and $\beta\equiv b\in(0,\infty)$.
\end{theorem}

Since the identity map $\mathrm{Id}\colon\mathbb{S}^n\to\mathbb{S}^n$ is a minimal immersion by first eigenfunctions, Theorem~\ref{thm:ev_exmp} implies that 
$$
\optV_1(\mathbb{S}^n,g_{\mathbb{S}^n}) = n\vol(\mathbb{S}^n,g_{\mathbb{S}^n}),
$$
and the supremum in the left hand side is achieved only for a constant density function. This is a direct generalization of Hersch's inequality for the $2$-sphere~\cite{Hersch}. 

Furthermore, the same assumption that $g_u\in [g]$ for some minimal immersion $u\colon M^n\to\mathbb{S}^k$ by the first eigenfunctions was used in~\cite{ESIconfvolume} to study $\Lambda_1(M^n,[g])$, where it is shown that 
$$
\Lambda_1(M,[g]) = n\left(\vol(M,g_u)\right)^{\frac{2}{n}},
$$
generalizing results of \cite{LiYau} from the two-dimensional setting. Combining this with Theorem~\ref{thm:ev_exmp} we obtain that for each $g\in [g_u]$
\begin{equation}
\label{ineq:ev_exmp}
\optV_1(M,g)\leq \Lambda_1(M,[g_u])\vol(M,g)^\frac{n-2}{n}.
\end{equation}
Finally, we remark that in addition to the identity map on the sphere there are many other examples of minimal immersions by first eigenfunctions, see~\cite{CES}. Inequality~\ref{ineq:ev_exmp} gives the value $\optV_1(M,g_u)$ for the corresponding metrics. 

\subsubsection{The Steklov problem.}
One 
fascinating feature of Laplace eigenvalue optimization on surfaces is its interaction with the optimization of Steklov eigenvalues, see e.g.~\cite{GL,GKL,KS1,KS2}. Below we present some analogous results on higher-dimensional manifolds.

Let $(\Omega^n,g)$ be a compact Riemannian manifold with boundary and $0\leq\rho\in C^\infty(\partial \Omega^n)$. Consider the weighted Steklov eigenvalues 
$$0=\sigma_0(\Omega^n,g,\rho)<\sigma_1(\Omega^n,g,\rho)<\ldots$$
corresponding to the weighted Dirichlet-to-Neumann problem
\begin{equation*}
\begin{cases}
\Delta_g u = 0  &\text{ in $\Omega$};\\
\bd_n u = \sigma\rho u &\text{ on $\bd \Omega$}.\\
\end{cases}
\end{equation*}
Similarly to~\eqref{def:optV} we define
\begin{equation*}
\optV^\bd_m(\Omega^n,g) = \sup_{0\leq\rho\in C^\infty(\bd\Omega^n)}\left\{ \sigma_m(\Omega^n,g,\rho)\int_{\bd\Omega^n} \rho\,ds_g\right\}.
\end{equation*}
If $n=2$, the problem of finding $\optV^\bd_m(\Omega^n,g)$ is equivalent to the classical conformal maximization problem on surfaces, see~\cite{KM}. For $n\geq 3$ this is no longer the case, but many well-known properties of the $2$-dimensional problem continue to hold for $\optV^\bd_m(\Omega^n,g)$. For example, in Proposition~\ref{prop:critical3} we show that densities critical for $\sigma_m(\rho)\|\rho\|_{L^1}$ naturally correspond to free boundary harmonic maps. This is a direct generalization of the $2$-dimensional result due to Fraser-Schoen~\cite{FS:extremal}, see also~\cite{KM}. Furthermore, it follows from~\cite[Theorem 1.12 and Theorem 5.2]{GKL} that
$$
\optV_m(M^n,g) = \sup_{\Omega^n\subset M^n}\optV^\bd_m(\Omega^n,g)
$$
and it follows from our Theorem~\ref{ev:main_theorem} that if $\Omega^n\subset M^n$ and $3\leq n\leq 5$ then
$$
\optV^\bd_1(\Omega^n,g) <\optV_1(M^n,g),
$$
which are higher-dimensional counterparts of~\cite[Corollary 1.5]{GKL} and~\cite[Theorem 1.5]{KS1} respectively.

\subsubsection{Negative eigenvalues of Schr\"odinger operators}
%
The quantities $\optV_m(M,g)$ can be equivalently defined using the following family of estimates introduced by  Grigor'yan-Netrusov-Yau \cite{GNY} and refined by Grigor'yan-Nadirashvili-Sire \cite{GNS}. On a closed Riemannian manifold $(M^n,g)$ of dimension $n\geq 3$, consider the class of Schr\"odinger operators $L_V$ of the form
$$L_V=\Delta_g-V=d^*d-V,$$
where $V\in L^{\infty}(M)$. In \cite{GNY}, it was shown that for nonnegative potentials $V\geq 0$, the number of negative eigenvalues $\mN(V)$ is bounded from below by the $L^1$-norm of the potential, raised to a suitable power; in \cite{GNS}, this was extended to arbitrary potentials, showing that
\begin{equation}\label{ind.lbd}
\mN(V)^{\frac{2}{n}}\geq C(M,g)\int_M V\,dv_g
\end{equation}
for any $V\in L^{\infty}(M)$. 

In particular, it follows that there is a uniform upper bound on the integral $\int_M V\,dv_g$ of the potential for all Schr\"odinger operators $L_V$ on $(M,g)$ with a given upper bound on the index $\mN(V)$. We show in 
Proposition~\ref{prop:equivalent} that
$$
\mathcal{V}_m(M,g):=\sup\left\{\left. \int_M V\,dv_g\right| \mN(V)\leq m\right\}.
$$
This formulation has some advantages over~\eqref{def:optV}, mainly related to the fact that there is no restriction on the sign of the potential $V$, see Section~\ref{sec:ev}. Thus, all the results described above can be equivalently reformulated on the language of Schr\"odinger operators, see e.g. Theorem~\ref{ev:main_theorem}.

\subsection{Ideas of the proofs}

The harmonic maps of Theorem~\ref{exthm} arise from a min-max construction generalizing those considered in \cite{Riv19, Riv20, KS1}, roughly along the lines suggested in the last paragraph of \cite[Section 7.1]{Sternpharm}. More precisely, fixing an isometric embedding $N\subset \mathbb{R}^L$ of the target manifold into some Euclidean space $\mathbb{R}^L$, we consider a Ginzburg--Landau type perturbation $E_{\epsilon}: W^{1,2}(M,\mathbb{R}^L)\to \mathbb{R}$ of the harmonic map problem as in \cite{CS} and \cite{LW99}, and show that there exist nonconstant critical maps $u_{\epsilon}\in C^{\infty}(M,\mathbb{R}^L)$ of Morse index $\ind_{E_{\epsilon}}(u_{\epsilon})\leq \ell+1$ for $E_{\epsilon}$, realizing a min-max energy
$$\mathcal{E}_{\epsilon}(M,g):=\inf_{(u_y)\in\Gamma}\max_{y\in \mathbb{B}^{\ell+1}}E_{\epsilon}(u_y).$$
Here, $\Gamma$ denotes the collection of continuous families of maps $\mathbb{B}^{\ell+1}\ni y\mapsto u_y\in W^{1,2}(M,\mathbb{R}^L)$ parametrized by the closed $(\ell+1)$-ball $\mathbb{B}^{\ell+1}$, such that the restriction to the boundary sphere $\mathbb{S}^{\ell}=\partial \mathbb{B}^{\ell+1}$ has the form
$$\mathbb{S}^{\ell}\ni y\mapsto u_y\equiv f(y)\in N,$$
for some homotopically nontrivial map $f\colon \mathbb{S}^{\ell}\to N$. Building on the analysis of \cite{LW99}, we establish lower-semicontinuity properties of the Morse index for critical points of $E_{\epsilon}$ in the limit $\epsilon\to 0$, and argue that if $N$ has no stable minimal two-spheres, then no energy is lost in the limit, and the min-max critical points $u_{\epsilon}$ of the perturbed functionals converge strongly in $W^{1,2}$ as $\epsilon\to 0$ to stationary harmonic maps $u\in W^{1,2}(M,\mathbb{R}^L)$ of Morse index $\ind_E(u)\leq \ell+1$. Observing that these harmonic maps $u$ must be locally stable near each point, we then appeal to the results of \cite{Hsu05} to obtain the refined partial regularity statement for these maps.

The proof of Theorem \ref{stab.etc} shares many features with the proof of the analogous result for surfaces in \cite{KS1}.  As in the 2-dimensional setting \cite{KS1}, the bound $\mathcal{V}_1(M,g)\leq 2E(u_k)$ follows in a fairly straightforward way from the min-max characterization of the maps $u_k$ in Corollary \ref{sphere.cor}, so to establish the equality $\mathcal{V}_1(M,g)=2E(u_k)$, the main challenge lies in showing that the Schr\"odinger operator $\Delta-|du_k|_g^2$ has only one negative eigenvalue for $k$ sufficiently large. And as in \cite{KS1}, to prove that $\mN(|du_k|_g^2)=1$, we first argue that the maps stabilize as $k$ becomes large, in the sense of Proposition~\ref{prop:intro_stab}.

The key difference in the higher-dimensional setting is in the formulation and proof of Proposition~\ref{prop:intro_stab}. While the two-dimensional stabilization result relies only on the uniform energy bound $\sup_k E(u_k)<\infty$ as $k\to\infty$, the stabilization on higher dimensional domains follows from the Morse index bound $\ind_E(u_k)\leq k+1$. In particular, by combining and refining arguments of \cite{ESindex}, \cite{SU84}, and \cite{LW06}, we show under the assumptions of Proposition~\ref{prop:intro_stab} that the 
harmonic maps $u_k\colon M^n\to \mathbb{S}^k$  satisfy $L^6$ gradient bounds $\|du_k\|_{L^6}\leq C$ independent of $k$, giving rise to compactness properties on the space of associated Schr\"odinger operators, from which the desired stabilization follows.

\subsection{Discussion}

In view of the connection to harmonic maps, it would be interesting to understand the existence theory for densities achieving $\mathcal{V}_m(M,g)$ with higher index $m\geq 2$. For $m=2$, one can approach the question by extending the min-max construction of \cite[Section 4]{KS1} to higher-dimensional manifolds; indeed, for $n\geq 3$, it is not difficult to show that this construction gives rise to harmonic maps $u_k\colon M^n\to \mathbb{S}^k$ with $2E(u_k)\geq \mathcal{V}_2(M,g)$ and index $\ind_E(u_k)\leq 2k+2$ (in contrast to dimension $n=2$, where a priori the min-max construction gives rise to a bubble tree for each $k$). However, the stabilization arguments from the proof of Theorem \ref{stab.etc} rely heavily on the asymptotic behavior $\limsup_{k\to\infty}\frac{\ind_E(u_k)}{2k}<1$ of the Morse index, and therefore do not carry over in a straightforward way to the case $m=2$. Indeed, it is quite possible that in general there is no density realizing $\mathcal{V}_m(M,g)$ for $m\geq 2$, generalizing non-existence phenomena for conformal maximizers of higher Laplace eigenvalues $\bar{\lambda}_m$ on surfaces.

More generally, with the basic analytic ingredients from  the proofs of Theorems \ref{exthm} and \ref{sphere.cor} in place, one can begin to ask more sophisticated questions about the space of harmonic maps from an arbitrary closed manifold into higher-dimensional spheres and other targets containing no stable minimal two-spheres. For instance, how many geometrically distinct harmonic maps can one find between a given closed manifold and the standard sphere $\mathbb{S}^k$ of dimension $k\geq 3$? Under what conditions on the domain $M$ and target $N$ can the a priori partial regularity of the maps in Theorem \ref{exthm} be improved? And to what extent can one relate the existence of harmonic maps with low Morse index from higher-dimensional manifolds into certain targets $N$ to other interesting geometric or topological features of $N$, as has been done for harmonic maps from the $2$-sphere \cite{SiuYau, MM88}?

\subsection*{Acknowledgements} The authors are grateful to Iosif Polterovich for remarks on a preliminary version of the manuscript. The research of M.K. is supported by the NSF grant DMS-2104254. The research of D.S. is supported by the NSF grant DMS-2002055.

\section{Existence and partial regularity of min-max harmonic maps}

\label{sec:minmax}

Let $(M^n,g)$ be a closed Riemannian manifold of dimension $n\geq 3$ and let $(N^k,h)$ be another closed manifold of dimension $k\geq 3$. Appealing to Nash's embedding theorem, we fix an isometric embedding
$$N\subset \mathbb{R}^L,$$
identifying $N$ henceforth with a submanifold of some high-dimensional Euclidean space $\mathbb{R}^L$. Denote by $\II_N$ the (vector-valued) second fundamental form for $N\subset \mathbb{R}^L$, defined by the convention
$$
\II_N(X,Y)=(D_XY)^{\perp},
$$
where $D$ is the usual Levi-Civita connection on $\mathbb{R}^L$ and $X,Y$ are tangent fields $X,Y\in \Gamma(TN)$.

A smooth map $u\colon M\to N$ is said to be \emph{harmonic} if it is a critical point of the energy functional
$$
E(u):=\frac{1}{2}\int_M |du|_g^2\,dv_g
$$
within the space of maps $C^{\infty}(M,N)$. Equivalently, $u\in C^{\infty}(M,N)$ is harmonic if and only if it satisfies the equation
\begin{equation}\label{wk.harm}
\Delta_g u+\langle \II(u),du^*du\rangle=0,
\end{equation}
where we write
$$du^*du:=\sum_{i=1}^L du^i\otimes du^i,$$
and denote by $\Delta_g$ the positive Laplacian $\Delta=d^*d$. An important consequence of \eqref{wk.harm} for $u\in C^{\infty}(M,N)$ is the fact that the stress-energy tensor
$$T_u:=\frac{1}{2}|du|_g^2g-du^*du$$
is divergence-free, or equivalently
\begin{equation}\label{stationary}
\int_M \frac{1}{2}|du|_g^2\mathrm{div}(X)-\langle du^*du,DX\rangle\,dv_g=0
\end{equation}
for every tangent vector field $X\in \Gamma(TM)$. Variationally, the condition \eqref{stationary} derives from the criticality of $u$ for $E$ with respect to variations of the form $u_t=u\circ \Phi_t$, where $\Phi_t\in \Diff(M)$ is a family of diffeomorphisms with $\Phi_0=\mathrm{Id}$.

The second variation of energy $E''(u)$ at a harmonic map $u\colon M\to N$ defines a quadratic form on the space 
$$\Gamma(u^*TN)=\{v\in C^{\infty}(M,\mathbb{R}^L)\mid v(x)\in T_{du(x)}N\}$$
of sections of the pullback bundle of $TN$. There are several ways to write $E''(u)$ in terms of the extrinsic or intrinsic geometry of $N\subset \mathbb{R}^L$; for our analytic purposes in this section, we simply note that
\begin{equation}\label{sec.var}
E''(u)(v,v)=\int_M |dv|^2-\langle \langle \II_N(u),du^*du\rangle, \II_N(u)(v,v)\rangle\,dv_g.
\end{equation}

In general, harmonic maps arising from variational constructions may be non-smooth, so it is necessary to extend these notions to maps which a priori lie only in the Sobolev space
$$W^{1,2}(M,N):=\{u\in W^{1,2}(M,\mathbb{R}^L)\mid u(x)\in N\text{ for a.e. }x\in M\}.$$
A map $u\in W^{1,2}(M,N)$ satisfying \eqref{wk.harm} weakly is said to be \emph{weakly harmonic}. Importantly, the maps we construct in this section will also satisfy the inner variation equation \eqref{stationary}; a map $u\in W^{1,2}(M,N)$ satisfying both \eqref{wk.harm} and \eqref{stationary} in the weak sense is known as a \emph{stationary harmonic} map to $N$. 

For a weakly harmonic map $u\in W^{1,2}(M,N)$, the second variation \eqref{sec.var} remains a well-defined quadratic form on the space
$$\mathcal{V}(u):=\{v\in [W^{1,2}\cap L^{\infty}](M,\mathbb{R}^L)\mid v(x)\in T_{u(x)}N\text{ a.e. }x\in M\}.$$
We define the \emph{Morse index} $\ind_E(u)$ of a weakly harmonic map $u\colon M \to N$ to be the index of $E''(u)$ as a quadratic form on $\mathcal{V}(u)$; i.e.,
$$\ind_E(u):=\max\{\dim V\mid V\subset \mathcal{V}(u),\text{ }E''(u)\text{ negative definite on }V\}.$$
It is sometimes convenient to refer to the index $\ind_E(u;\Omega)$ of $u$ on a given domain $\Omega\subset M$; by this we simply mean the index of $E''(u)$ restricted to variations supported in $\Omega$--i.e.,
$$\ind_E(u;\Omega):=\{\dim V\mid V\subset \mathcal{V}(u)\cap W_0^{1,2}(\Omega,\mathbb{R}^L),\text{ }E''(u)|_V<0\}.$$
A harmonic map $u\colon M\to N$ is said to be \emph{stable} if $\ind_E(u)=0$, and we say that $u$ is stable in a given domain $\Omega\subset M$ if $\ind_E(u;\Omega)=0$. We will say that $u\colon M\to N$ is \emph{locally stable} if every point $p\in M$ has a neighborhood $U\ni p$ on which $u$ is stable. For smooth harmonic maps, we may of course replace $\mathcal{V}(u)$ with the space of smooth sections $\Gamma(u^*TN)$ in all these definitions.

\subsection{The regularized min-max construction}\label{pert.minmax.sec}

The proof of Theorem \ref{exthm} is based on a min-max construction generalizing those studied in \cite{Sternpharm, KS1, Riv19, Riv20}. As in those cases, rather than attempting to apply variational methods directly to the Dirichlet energy on the highly nonlinear space $W^{1,2}(M,N)$, we consider a regularized construction based on the Ginzburg-Landau type functionals studied in \cite{CS,LW99}, whose definition we recall below.

We continue to view our target $(N^k,h)$ as an isometrically embedded submanifold $N\subset \mathbb{R}^L$ in some large-dimensional Euclidean space.
Fix a positive number $\delta_0(N)>0$ such that the squared distance to $N$ 
$$d_N^2\colon B_{\delta_0}(N)\to \mathbb{R}$$ 
is smooth on the $\delta_0$-neighborhood $B_{\delta_0}(N)\subset \mathbb{R}^L$, and the nearest point projection 
$$\Pi\colon B_{\delta_0}(N)\to N$$
is smooth and well-defined. Let $R_0>\delta_0$ be a large radius such that $N\subset B_{R_0}(0)$. Now, fix a smooth potential function $W: \mathbb{R}^L\to [0,\infty)$ such that
$$W(a)=d_N^2(a)\text{ for }a\in B_{\delta_0/2}(N),$$
$$W(a)=|a|^2\text{ for }a\in \mathbb{R}^L\setminus B_{R_0}(0),$$ 
and
$$\frac{\delta_0^2}{4}\leq W(a)\leq R_0^2\text{ on }B_{R_0}(0)\setminus B_{\delta_0/2}(L).$$
Note that we then have
\begin{equation}\label{w.dist.comp}
c_1\leq \frac{W(a)}{d_N(a)^2}\leq C_1\text{ for some constants }0<c_1<C_1<\infty.
\end{equation}
Following~\cite{CS}, we then define a family of energy functionals
$$E_{\epsilon}: W^{1,2}(M,\mathbb{R}^L)\to \mathbb{R}$$
by 
\begin{equation}
E_{\epsilon}(u):=\int_M\frac{1}{2}|du|_g^2+\frac{W(u)}{\epsilon^2}\,dv_g,
\end{equation}
whose critical points satisfy the Euler-Lagrange system
\begin{equation}\label{gl.eqn}
\Delta u+\frac{DW(u)}{\epsilon^2}=0.
\end{equation}
For potentials $W(u)$ of the form given above, it is easy to check that weak solutions $u\in W^{1,2}(M,\mathbb{R}^L)$ of \eqref{gl.eqn} are always smooth.

Now, given a homotopically nontrivial map $f\colon \mathbb{S}^{\ell}\to N$, we define a collection $\Gamma_f\subset C^0(B^{\ell+1},W^{1,2}(M,\mathbb{R}^L))$ of continuous families of maps in $W^{1,2}(M,\mathbb{R}^L)$ parametrized by the closed unit $(\ell+1)$-ball by
$$\Gamma_f:=\{\mathbb{B}^{\ell+1}\ni y\mapsto u_y\in W^{1,2}(M,\mathbb{R}^L)\mid u_y\equiv f(y)\text{ for }y\in \mathbb{S}^{\ell}\}.$$
Then define the associated min-max energy
$$\mathcal{E}_{f,\epsilon}:=\inf_{(u_y)\in \Gamma_f}\max_{y\in B^{\ell+1}}E_{\epsilon}(u_y).$$
It is straightforward to check that $\mathcal{E}_{f,\epsilon}=\mathcal{E}_{f',\epsilon}$ if $f\simeq f'$, so that the construction depends on $f$ only through its homotopy class--and, it should be noted, maps of distinct homotopy types can also give rise to the same min-max construction. 

Observing that $\mathcal{E}_{f,\epsilon}$ is a decreasing function of $\epsilon$, we define
$$\mathcal{E}_f(M,g):=\lim_{\epsilon\to 0}\mathcal{E}_{f,\epsilon}(M,g)=\sup_{\epsilon>0}\mathcal{E}_{f,\epsilon}(M,g).$$
As a first step to ensure that $\mathcal{E}_f(M,g)$ could be realized as the energy of some nonconstant harmonic map, we need to check that $0<\mathcal{E}_f(M,g)<\infty$. 

\begin{lemma} The limiting energy $\mathcal{E}_f=\lim_{\epsilon\to 0}\mathcal{E}_{f,\epsilon}$ satisfies the lower bound
\begin{equation}\label{unif.en.bds}
\mathcal{E}_f(M,g)\geq \frac{1}{2}\delta_0(N)^2\vol(M,g)\lambda_1(M,g)>0.
\end{equation}.
\end{lemma}
\begin{proof}
Given a family $\mathbb{B}^{\ell+1}\ni y \mapsto u_y\in W^{1,2}(M,\mathbb{R}^L)$ in $\Gamma_f$, observe that the map
$$\alpha\colon \mathbb{B}^{\ell+1}\to \mathbb{R}^L$$
given by taking averages
$$\alpha(y):=\frac{1}{\vol(M,g)}\int_Mu_y\,dv_g$$
cannot have image contained in the tubular neighborhood $B_{\delta_0}(N)$. Indeed, if we did have $\alpha(\mathbb{B}^{\ell+1})\subset B_{\delta_0}(N)$, then by postcomposing with the nearest-point projection $\Pi_N\colon B_{\delta_0}(N)\to N$, we would obtain a map 
$$\Pi_N\circ \alpha\colon \mathbb{B}^{\ell+1}\to N$$
continuously extending the map $f=\alpha|_{\mathbb{S}^{\ell}}\colon \mathbb{S}^{\ell}\to N$ to the ball $\mathbb{B}^{\ell+1}$, violating the assumption that $f$ is homotopically nontrivial. 

Thus, for any family $(u_y)\in \Gamma_f$, there must be some $y\in \mathbb{B}^{\ell+1}$ for which
$$\alpha(y)=\frac{1}{\vol(M)}\int_Mu_y \,dv_g\notin B_{\delta_0}(N).$$
For this $y$, the triangle inequality gives the pointwise inequality
$$\delta_0\leq d_N(\alpha(y))\leq d_N(u_y)+|u_y-\alpha(y)|,$$
and integrating over $M$ gives
\begin{equation*}
\begin{split}
&\delta_0(N)\vol(M)\leq \int_Md_N(u_y)\,dv_g+\int_M|u_y-\alpha(y)|\,dv_g\\
&\leq \vol(M)^{1/2}\left(\int d_N(u_y)^2\,dv_g\right)^{1/2}+\vol(M)^{1/2}\|u_y-\alpha(y)\|_{L^2}\\
&\leq \vol(M)^{1/2}\left(\int C_1W(u_y)^2\,dv_g\right)^{1/2}+\vol(M)^{1/2}\lambda_1(M,g)^{-1/2}\|du\|_{L^2},
\end{split}
\end{equation*}
where in the final line we've used \eqref{w.dist.comp} and the Poincar\'e inequality 
$$
\lambda_1(M,g)\|u_y-\alpha(y)\|_{L^2}^2\leq \|du\|_{L^2}^2.
$$
 Recalling the definition of $E_{\epsilon}$, we then see that
$$\delta_0(N)\vol(M)^{1/2}\leq C_1^{1/2}\epsilon E_{\epsilon}(u_y)^{1/2}+\left(\lambda_1(M,g)^{-1/2}2E_{\epsilon}(u_y)\right)^{1/2},$$
and since the family $(u_y)\in \Gamma_f$ was arbitrary, it follows that
$$\delta_0(N)\vol(M)^{1/2}\leq \left(C_1^{1/2}\epsilon+\sqrt{2/\lambda_1(M,g)}\right)\mathcal{E}_{f,\epsilon}^{1/2}.$$
Squaring both sides and taking the liminf as $\epsilon\to 0$ then gives
$$\delta_0(N)^2\vol(M)\leq \frac{2}{\lambda_1(M,g)}\mathcal{E}_f,$$
as desired.
\end{proof}

\begin{lemma} There exists $C>0$ independent of $\epsilon>0$ such that
\label{lemma:energy_upbound}
$$\mathcal{E}_{f,\epsilon}(M,g)\leq C$$
for all $\epsilon\in (0,1)$. In particular, $\mathcal{E}_f(M,g)<\infty$.
\end{lemma}
\begin{proof} To prove the desired upper bounds, it suffices to produce a single family $(u_y)\in \Gamma_f$ for which the energy $E_{\epsilon}$ is uniformly bounded independent of $\epsilon$. There should be many ways to construct such a family; here we follow a construction similar to that of \cite[Section 3]{SternGL}.

By classical results on triangulations (see, e.g., \cite{Wh}), there exists a finite simplicial complex $\mathcal{K}$ in some Euclidean space $\mathbb{R}^m$ (take $m\geq \ell+1$, without loss of generality) and a bi-Lipschitz map 
$$\Phi\colon M\to |\mathcal{K}|$$
from $M$ to the underlying space $|\mathcal{K}|$ of $\mathcal{K}$. Note that for a given $n$-dimensional subspace $V\subset \mathbb{R}^m$ and a generic $(\ell+1)$-dimensional subspace $\Pi\subset \mathbb{R}^m$, the projection $P\colon V\to \Pi$ has full rank $\dim(P(V))=\min\{n,\ell+1\}$. In particular, for the finite simplicial complex $\mathcal{K}$, by projecting onto a generic $(\ell+1)$-dimensional subspace of $\mathbb{R}^m$, we can find a linear map
$$P\colon \mathbb{R}^m\to \mathbb{R}^{\ell+1}$$
such that the restriction of $P$ to each $n$-simplex $\Delta \in \mathcal{K}$ has maximal rank $\min\{n,\ell+1\}$. Denote by $\rho\colon \mathbb{R}^{\ell+1}\to \mathbb{S}^{\ell}$ the radial retraction $\rho(x):=x/|x|$. Now, with $f\colon\mathbb{S}^{\ell}\to N$ as above, we can define a family of maps $\mathbb{B}^{\ell+1}\ni y\mapsto u_y\in W^{1,2}(M, N)$ by
$$u_y(x):=f(\rho(P(\Phi(x))+(1-|y|)^{-1}y)).$$
Note that the map $u_y$ is locally Lipschitz away from the preimage $(P\circ \Phi)^{-1}(\{-(1-|y|)^{-1}y\})$ of codimension $\geq \min\{n,\ell+1\}\geq 3$.
A priori the family is defined only for $|y|<1$, but it is straightforward to check that the assignment extends to a continuous family $\mathbb{B}^{\ell+1}\ni y\mapsto u_y\in L^p(M,N)$ on the closed unit ball $\mathbb{B}^{\ell+1}$ with $u_y\equiv f(y)$ for $y\in \mathbb{S}^{\ell}$. We argue next that the Dirichlet energy $E(u_y)$ is bounded independent of $y\in \mathbb{B}^{\ell+1}$.

Indeed, since the map $f\in C^1(\mathbb{S}^{\ell},N)$, the bi-Lipschitz identification $\Phi\in \mathrm{Lip}(M,|\mathcal{K}|)$, and the linear map $P\colon\mathbb{R}^m\to \mathbb{R}^{\ell+1}$ are all fixed independent of $y$, we see that
$$|du_y|(x)\leq C|d\rho(P(\Phi(x))+(1-|y|)^{-1}y)|\leq C'\frac{1}{|P(\Phi(x))+(1-|y|)^{-1}y|},$$
so (writing $y'=-(1-|y|)^{-1}y$), we simply need to check the upper bound
\begin{equation}\label{obv.en.bd}
\int_M |P(\Phi(x))-y'|^{-2}\,dv_g\leq C
\end{equation}
independent of $y'\in \mathbb{R}^{\ell+1}$. Since the map $\Phi\colon M\to |\mathcal{K}|$ is bi-Lipschitz, we clearly have
$$\int_M|P(\Phi(x))-y'|^{-2}\leq C'\int_{|\mathcal{K}|}|P(x')-y'|^{-2}d\mathcal{H}^n(x')=C'\sum_{\Delta\in \mathcal{K}}\int_{\Delta}|P(x')-y'|^{-2}.$$

Since the affine map $P\colon\Delta \to \mathbb{R}^{\ell+1}$ has full rank $s:=\min\{n,\ell+1\}$ on each $\Delta\in \mathcal{K}$, it follows from the area formula when $n\leq \ell+1$ or the coarea formula when $n>\ell+1$ that
$$\int_{\Delta}|P(x')-y'|^{-2}d\mathcal{H}^n(x')\leq C(\Delta)\int_{P(\Delta)}|w-y'|^{-2}\mathcal{H}^{n-s}(P^{-1}\{w\}\cap\Delta)d\mathcal{H}^s(w),$$
and since $P$ restricts to an affine linear map of rank $s$ on the bounded $n$-simplex $\Delta$, there is clearly a uniform bound $\mathcal{H}^{n-s}(P^{-1}\{w\}\cap \Delta)\leq C'(\Delta)$, so that
$$\int_{\Delta}|P(x')-y'|^{-2}d\mathcal{H}^n(x')\leq C''(\Delta)\int_{P(\Delta)}|w-y'|^{-2}d\mathcal{H}^s(w).$$
Choosing $R>0$ such that $P(|\mathcal{K}|)\subset B^{\ell+1}_R(0)$, we see that $P(\Delta)$ is contained in the intersection of $B_R^{\ell+1}(0)$ with a subspace of dimension $s=\min\{n,\ell+1\}\geq 3$, and consequently
$$\int_{P(\Delta)}|w-y'|^{-2}d\mathcal{H}^s(w)\leq \int_{B^s_R(0)}|w|^{-2}dw\leq CR^{s-2}.$$
Finally, summing over the finite collection of $n$-simplices $\Delta\in \mathcal{K}$, we see that
$$\int_M|P(\Phi(x))-y'|^{-2}\,dv_g\leq \sum_{\Delta\in \mathcal{K}}C R^{s-2}=C_0,$$
as desired.

We've now seen that the family $\mathbb{B}^{\ell+1}\ni y\mapsto u_y\in W^{1,2}(M,N)$ is at least weakly continuous, with energy bounded above independent of $y\in \mathbb{B}^{\ell+1}$. Arguing as in \cite{KS1}, we now consider the convolution
$$u_y^t(x):=\int_MK_t(x,x')u_y(x')\,dv_g(x')$$
of $u_y$ with the heat kernel $K_t$ of $(M,g)$, so that, for each $t>0$, the assignment $\mathbb{B}^{\ell+1}\ni u_y^t\in W^{1,2}(M,\mathbb{R}^L)$ is continuous in $W^{1,2}$, and defines in particular a test family $(u_y^t)\in \Gamma_f$ for the min-max construction. Moreover, as in \cite{KS1}, it is not difficult to see that
$$\limsup_{t\to 0}\max_{y\in \mathbb{B}^{\ell+1}}E_{\epsilon}(u_y^t)\leq \sup_{y\in \mathbb{B}^{\ell+1}}\int \frac{1}{2}|du_y|_g^2\,dv_g\leq C(M)<\infty.$$
Hence, for each $\epsilon>0$, there exists $t_{\epsilon}>0$ such that the family $(u_y^{t_{\epsilon}})\in \Gamma_f$ satisfies
$$\mathcal{E}_{f,\epsilon}\leq \max_{y\in \mathbb{B}^{\ell+1}}E_{\epsilon}(u_y^{t_{\epsilon}})\leq C'(M),$$
giving the desired energy bound.
\end{proof}

Next, to ensure that the energies $\mathcal{E}_{f,\epsilon}(M,g)$ are actually achieved by critical points of $E_{\epsilon}$, we observe that the energies $E_{\epsilon}$ are sufficiently regular functionals satisfying a Palais-Smale condition. 

\begin{proposition} The functionals $E_{\epsilon}\colon W^{1,2}(M,\mathbb{R}^L)\to \mathbb{R}$ are $C^2$, with first and second derivatives
$$\langle E_{\epsilon}'(u),v\rangle=\int_M\langle du,dv\rangle+\epsilon^{-2}\langle DW(u),v\rangle\,dv_g$$
and
$$E_{\epsilon}''(u)(v,v)=\int_M|dv|^2+\epsilon^{-2}\langle D^2W(u),v\otimes v\rangle\,dv_g,$$
where $D^2W$ denotes the Hessian of $W$.
The operator associated to $E_{\epsilon}''(u)$ is Fredholm at any solution $u$ of \eqref{gl.eqn}. Moreoever, $E_{\epsilon}$ satisfies the Palais-Smale compactness condition: for any sequence $u_j\in W^{1,2}(M,\mathbb{R}^L)$ such that 
$$
\sup_j E_{\epsilon}(u_j)<\infty\quad \text{and}\quad \lim_{j\to\infty}\|E_{\epsilon}'(u_j)\|_{(W^{1,2})^*}=0,
$$
 there exists a subsequence converging strongly in $W^{1,2}(M,\mathbb{R}^L)$.
\end{proposition}
\begin{proof} Most of the statements can be checked by direct computation; the proof of the Palais--Smale condition, though standard, is less trivial, so we include it for the convenience of the reader. To verify the Palais--Smale condition, let $u_j\in W^{1,2}(M,\mathbb{R}^L$ be a sequence satisfying $E_{\epsilon}(u_j)\leq C$ and 
\begin{equation}\label{ps.hyp}
\left|\int_M \langle du_j,dv\rangle+\epsilon^{-2}\langle DW(u_j),v\rangle\,dv_g\right|\leq \delta_j\|v\|_{W^{1,2}}
\end{equation}
for some sequence $\delta_j\to 0$. Since the potential function $W(a)$ agrees with $|a|^2$ outside of a compact subset of $\mathbb{R}^L$, it is easy to see that boundedness of $E_{\epsilon}(u_j)$ implies boundedness of $\|u_j\|_{W^{1,2}}$, and therefore we can find a subsequence (unrelabelled) and a map $u\in W^{1,2}(M,\mathbb{R}^L)$ for which 
$$u_j\to u\text{ weakly in }W^{1,2}(M,\mathbb{R}^L)\text{ and strongly in }L^2(M,\mathbb{R}^L).$$
Since energy is lower semi-continuous under weak convergence, we have
\begin{equation*}
\begin{split}
\lim_{j\to\infty}\int |d(u-u_j)|^2 & =\lim_{j\to\infty}\int |du|^2+|du_j|^2-2\langle du,du_j\rangle \\
&\leq \lim_{j\to\infty}2\int (|du_j|^2-\langle du,du_j\rangle),
\end{split}
\end{equation*}
so to show that $u_j\to u$ strongly in $W^{1,2}(M,\mathbb{R}^L)$, it suffices to show that
\begin{equation}\label{cvg.suff}
\lim_{j\to\infty}\int_M |du_j|^2-\langle du_j,du\rangle\,dv_g=0.
\end{equation}
To this end, note that $u_j-u$ is bounded in $W^{1,2}$, so taking $v=u_j-u$ in \eqref{ps.hyp} gives
$$
\lim_{j\to\infty}\int_M |du_j|^2-\langle du_j,du\rangle+\epsilon^{-2}\langle DW(u_j),u_j-u\rangle\,dv_g=0.
$$
Moreover, we know that $DW(u_j)$ is bounded in $L^2$ and $u_j-u\to 0$ in $L^2$, so
$$
\lim_{j\to\infty}\int_M \epsilon^{-2}\langle DW(u_j),u_j-u\rangle\,dv_g=0
$$
as well, and \eqref{cvg.suff} follows.

\end{proof}

With the preceding ingredients in place, standard results in critical point theory (see, e.g., \cite[Chapter 10]{Ghou}) imply that the min-max energies $\mathcal{E}_{f,\epsilon}$ are indeed achieved by critical points with Morse index bounded above by the number of parameters in the construction.

\begin{proposition}\label{pert.minmax} For each $\epsilon>0$, there exists a critical point $u_{\epsilon}\colon M\to \mathbb{R}^L$ for $E_{\epsilon}$ of Morse index $\ind_{E_{\epsilon}}(u_{\epsilon})\leq \ell+1$ and
$$
E_{\epsilon}(u_{\epsilon})=\mathcal{E}_{f,\epsilon}(M,g).
$$
\end{proposition}

Note that nowhere in the preceding subsection have we invoked any assumptions on the geometry of the target $N$, beyond the topological condition that $\pi_{\ell}(N)\neq 0$. The assumption that $N$ admits no stable minimal two-spheres enters in the next section, where it plays a crucial role in ensuring strong compactness as $\epsilon\to 0$ for families of critical points $u_{\epsilon}$ of $E_{\epsilon}$ like those given by Proposition \ref{pert.minmax}.

\subsection{Convergence to stationary harmonic maps}

The goal of this subsection is to prove the following compactness result, showing that families of critical points $u_{\epsilon}$ of $E_{\epsilon}$ like those given by Proposition \ref{pert.minmax} converge strongly to stationary harmonic maps as $\epsilon\to 0$, provided $N$ satisfies the crucial hypothesis that every stable harmonic map $\phi\colon \mathbb{S}^2\to N$ is constant.

\begin{theorem}\label{stat.lim.thm} Let $(M^n,g)$ be a closed Riemannian manifold of dimension $n\geq 3$ and $N^k\subset \mathbb{R}^L$ a closed manifold of dimension $k$ as above, such that every stable harmonic map $\phi\colon \mathbb{S}^2\to N$ is constant. Let $u_{\epsilon}\colon M\to \mathbb{R}^L$ be a family of critical points for $E_{\epsilon}$ with $E_{\epsilon}(u_{\epsilon})\leq C<\infty$ and $\ind_{E_{\epsilon}}(u_{\epsilon})\leq I_0<\infty$ as $\epsilon\to 0$. Then $u_{\epsilon}$ converges strongly in $W^{1,2}$ to a stationary harmonic map $u\colon M\to N$ of Morse index $\ind_E(u)\leq I_0$.
\end{theorem}

Under the stronger condition that $N$ admits \emph{no} nonconstant harmonic maps $\mathbb{S}^2\to N$, we could appeal directly to \cite[Theorem A, Corollary B]{LW99} to deduce the strong convergence of the critical points $u_{\epsilon}$ to a stationary harmonic map. In this case, however, $\pi_{\ell}(N)=0$ for every $\ell\geq 3$, so one cannot produce nontrivial families of critical points via the methods of the previous section. To prove Theorem \ref{stat.lim.thm}, we revisit the analysis of \cite{LW99} of possible energy concentration for the maps $u_{\epsilon}$, observing that the added assumption of bounded Morse index allows us to rule out energy concentration more generally in the absence of \emph{stable} harmonic maps $\mathbb{S}^2\to N$.

As in \cite{LW99}, the first ingredients needed to begin the blow-up analysis are the monotonicity formula and a small-energy regularity theorem for solutions of \eqref{gl.eqn}, which we recall below.

\begin{lemma}\label{mono}
Let $u\colon M\to \mathbb{R}^L$ be a critical point for $E_{\epsilon}$. Then on geodesic balls $B_r(p)$ of radius $r<\inj(M)$, we have
$$\frac{d}{dr}\left(e^{Cr^2}r^{2-n}\int_{B_r(p)}e_{\epsilon}(u)\right) \geq e^{Cr^2}r^{2-n}\left(\int_{\partial B_r(p)}\left|\frac{\partial u_{\epsilon}}{\partial \nu_p}\right|^2+\frac{2}{r}\int_{B_r(p)}\frac{W(u_{\epsilon})}{\epsilon^2}\right),$$
where $\nu_p$ denotes the gradient of the distance function $d_p$ to $p$, and $C=C(n,k)$ is a constant depending on the geometry of $(M,g)$ only through the dimension $n=\dim M$ and a sectional curvature bound $k\geq |\mathrm{sec}(M,g)|$.
\end{lemma}

\begin{lemma}\label{eps.reg} There exist constants $C(N,n,k)<\infty$, $\eta_0(N,n,k)>0$ depending only on the target manifold $N\subset \mathbb{R}^L$, the dimension $n=\dim M$, and a sectional curvature bound $|\mathrm{sec}(M,g)|\leq k$, such that if $u\colon B_r(p)\to N$ solves \eqref{gl.eqn} on a ball $B_{2r}(p)\subset M$ with $2r<\min\{\inj(M,g),1\}$, and
$$r^{2-n}E_{\epsilon}(u;B_{2r}(p))<\eta_0,$$
then $r^2 e_{\epsilon}(u)\leq 1$ on $B_{r/2}(p)$.
\end{lemma}

Though both the monotonicity and small-energy regularity results are well-known to experts--perhaps in a slightly different form--we include proofs in the appendix for the convenience of the reader.

Next, we recall the main result of \cite{LW99}, which shows that families of critical points $u_{\epsilon}$ of $E_{\epsilon}$ with bounded energy exhibit limiting behavior as $\epsilon\to 0$ similar to that of stationary harmonic maps, as described in \cite{Lin99}. Technically speaking, their results don't apply directly to the functionals $E_{\epsilon}$ as defined in the previous subsection, since in \cite{LW99} the potential $W(a)$ is taken to be constant outside a tubular neighborhood of $N$, while we choose to set $W(a)=|a|^2$ outside of a compact set (simply for convenience in formulating the Palais-Smale property); however, it is easy to check that this has no effect on the analysis, since the potentials $W(a)$ agree on a tubular neighborhood of $N$.

\begin{theorem}[Lin, Wang~\cite{LW99}\label{lw.thm}]

If $u_{\epsilon}\colon M\to \mathbb{R}^L$ is a family of critical points for $E_{\epsilon}$ with $E_{\epsilon}(u_{\epsilon})\leq C$, then there exists a weakly harmonic map $u\in W^{1,2}(M,N)$ and a closed, $(n-2)$-rectifiable set $\Sigma\subset M$ such that, after passing to a subsequence, $u_{\epsilon}\to u$ in $C^{\infty}_{loc}(M\setminus\Sigma)$ and $u_{\epsilon}\to u$ weakly in $W^{1,2}(M,\mathbb{R}^L)$, and the discrepancy measure
$$
\nu:=\lim_{\epsilon\to 0}e_{\epsilon}(u_{\epsilon})dv_g-\frac{1}{2}|du|^2dv_g
$$
has the form
$$
\nu=\theta(x) \mathcal{H}^{n-2}\lfloor \Sigma.
$$
\end{theorem}

\begin{remark}\label{bubble.rk} Moreover, it follows from the analysis of \cite{LW99} that for $\nu$-a.e. $x\in M$, there exists a sequence of points $x_\epsilon\to x$ and of scales $r_{\epsilon}\to 0$ with $\lim_{\epsilon\to 0}\frac{\epsilon}{r_{\epsilon}}\to 0$ such that the rescaled maps
$$
\tilde{u}_{\epsilon}\colon \mathbb{R}^n\to \mathbb{R}^L
$$
given by
$$\tilde{u}_{\epsilon}(y)=u_{\epsilon}(\exp_{x_\epsilon}(y/r_{\epsilon}))$$
converge in $C^{\infty}_{loc}(\mathbb{R}^n)$ to a harmonic map $\tilde{u}\colon \mathbb{R}^n\to N$ of the form (up to rotations)
$$u(x_1,\ldots,x_n)=\phi(x_1,x_2),$$
where $\phi\colon \mathbb{R}^2\to N$ is a smooth, nonconstant, finite-energy harmonic map. In particular, $\phi$ may be identified with a smooth harmonic map $\phi\colon \mathbb{S}^2\to N$.
\end{remark}

In addition to the analysis of \cite{LW99}, another key ingredient required to prove Theorem \ref{stat.lim.thm} is the following lemma, establishing \emph{lower semi-continuity} of Morse index for a family of critical points $u_{\epsilon}$ of $E_{\epsilon}$ satisfying the hypotheses of Theorem \ref{stat.lim.thm}.

\begin{lemma}\label{ind.lim} Let $u_{\epsilon}\in C^{\infty}(M,\mathbb{R}^L)$ be a family of critical points for $E_{\epsilon}$ with $E_{\epsilon}(u_{\epsilon})\leq C$ and $\ind_{E_{\epsilon}}(u_{\epsilon})\leq I_0$. Then the weakly harmonic map $u\in W^{1,2}(M,N)$ in the conclusion of Theorem \ref{lw.thm} has Morse index $\ind_E(u)\leq I_0$, in the sense that any vector subspace
$$V\subset \{v\in [W^{1,2}\cap L^{\infty}](M,\mathbb{R}^L)\mid v(x)\in T_{u(x)}N\text{ for a.e. }x\in M\}$$
on which
$$
E''(u)(v,v)=\int_M |dv|^2-\langle \II_N(u)(du(e_i),du(e_i)),\II_N(u)(v,v)\rangle\,dv_g
$$
is negative definite must have dimension $\dim(V)\leq I_0$.
\end{lemma}

\begin{proof} Let $u\in W^{1,2}(M,N)$ be the weakly harmonic map arising as the weak limit of the critical points $u_{\epsilon}\in C^{\infty}(M,\mathbb{R}^L)$, and let 
$$V\subset \{v\in [W^{1,2}\cap L^{\infty}](M,\mathbb{R}^L)\mid v(x)\in T_{u(x)}N\text{ for a.e. }x\in M\}$$
be a finite-dimensional subspace on which $E''(u)$ is negative definite. Since $V$ is finite-dimensional, we can then find $\beta>0$ such that
\begin{equation}\label{neg.def}
E''(u)(v,v)\leq -\beta(\|v\|_{L^{\infty}}^2+\|dv\|_{L^2}^2)\text{ for all }v\in V.
\end{equation}

By Theorem \ref{lw.thm}, there exists a closed, $(n-2)$-rectifiable set $\Sigma\subset M$ such that the convergence $u_{\epsilon}\to u$ is $C^{\infty}$ on any compact subset of $M\setminus \Sigma$. We argue next that without loss of generality the space $V$ of energy-decreasing deformations can be taken to have support in $M\setminus \Sigma$. Indeed, since $\Sigma$ is $(n-2)$-rectifiable, the $r$-neighborhoods $B_r(\Sigma)$ satisfy
\begin{equation}\label{mink.bd}
\vol(B_r(\Sigma))\leq Cr^2
\end{equation}
for a constant $C(\Sigma)$ independent of $r$, and for any $\delta\in (0,1)$, standard computations show that the logarithmic cutoff function $\phi_{\delta}$ supported in $B_{\delta}(\Sigma)$ given by 
$$\phi_{\delta}\equiv 1\text{ on }B_{\delta^2}(\Sigma)$$
and
$$\phi_{\delta}(x)=\frac{\log(\dist_{\Sigma}(x)/\delta)}{\log(\delta)}\text{ on }B_{\delta}(\Sigma)\setminus B_{\delta^2}(\Sigma)$$
satifies an estimate of the form
\begin{equation}\label{phi.en.bd}
\|d\phi_{\delta}\|_{L^2}^2\leq \frac{C(\Sigma)}{|\log\delta|}.
\end{equation}

Writing $\psi_{\delta}=1-\phi_{\delta}$, define
$$
V_{\delta}:=\{\psi_{\delta} v\mid v\in V\},
$$
so that $\supp(v)\subset M\setminus B_{\delta^2}(\Sigma)$ for every $v\in V_{\delta}$. Since $V$ is finite-dimensional and $\psi_{\delta}\to 1$ as $\delta\to 0$, it is easy to see that $\dim (V_{\delta})=\dim(V)$ for $\delta$ sufficiently small. Moreover, for every $v\in V$, we have
\begin{equation*}
\begin{split}
&E''(u)(\psi_{\delta}v,\psi_{\delta}v)-E''(u)(v,v)\leq \\
&\leq\int_M |v|^2|d\phi_{\delta}|^2+2|dv||v||d\phi_{\delta}|\,dv_g
+\|II_N\|_{L^{\infty}(N)}^2\int_{B_{\delta}(\Sigma)}|du|^2|v|^2\,dv_g \\
&\leq \frac{C}{|\log\delta|}\|v\|_{L^{\infty}}^2+\frac{C}{\sqrt{|\log\delta|}}\|v\|_{L^{\infty}}\|dv\|_{L^2}
+C\|v\|_{L^{\infty}}^2\int_{B_{\delta}(\Sigma)}|du|^2\,dv_g.
\end{split}
\end{equation*}
Now, since $u\in W^{1,2}$, certainly $\lim_{\delta\to 0}\int_{B_{\delta}(\Sigma)}|du|^2=0$, and it follows from the preceding computation that
$$E''(u)(\psi_{\delta}v,\psi_{\delta}v)\leq E''(u)(v,v)+\frac{\beta}{2}(\|v\|_{L^{\infty}}^2+\|dv\|_{L^2}^2)$$
for all $v\in V$ for $\delta>0$ sufficiently small. Together with \eqref{neg.def}, this gives
\begin{equation}\label{delta.neg}
E''(u)(v,v)\leq -\frac{\beta}{2}(\|v\|_{L^{\infty}}^2+\|dv\|_{L^2}^2)\text{ for all }v\in V_{\delta}
\end{equation}
for $\delta>0$ sufficiently small. For the remainder of the proof, let us fix some $\delta>0$ for which \eqref{delta.neg} holds.

On the domain $\Omega=M\setminus B_{\delta^2}(\Sigma)$, we have smooth convergence of the maps $u_{\epsilon}\in C^{\infty}(\Omega,\mathbb{R}^L)$ to the strongly harmonic map $u\in C^{\infty}(\Omega,N)$. For $\epsilon>0$ small enough, it follows that, on $\Omega$, $u_{\epsilon}$ takes values in a tubular neighborhood of $N$ where the potential $W$ is given by $d_N(\cdot)^2$, so that
\begin{equation}\label{dw.comp}
DW(u_{\epsilon})=2d_N(u_{\epsilon})(D d_N)(u_{\epsilon})
\end{equation}
and, by Lemma \ref{dsquared.lem} in the appendix,
\begin{equation}\label{hess.w.exp}
|D^2W(u_{\epsilon})-2P^{\perp}(u_{\epsilon})+\langle 2d_N(u_{\epsilon})(Dd_N)(u_{\epsilon}),B(u_{\epsilon})\rangle|\leq CW(u).
\end{equation}
Here we use the notation of Section \ref{app.1}, where $P(u_{\epsilon})\in \mathrm{End}(\mathbb{R}^L)$ denotes the projection onto the tangent space $T_{\Pi_N(u_{\epsilon})}N\subset \mathbb{R}^L$ of the nearest point $\Pi_N(u_{\epsilon})$ to $u_{\epsilon}$ in $N$, $P^{\perp}(u_{\epsilon})=I-P(u_{\epsilon})$ denotes the projection onto the complement, and $B(u_{\epsilon})$ is the $\mathbb{R}^L$-valued two-tensor given by 
$$B(u_{\epsilon})(v,w):=\II_N(\Pi(u_{\epsilon}))(P(u_{\epsilon})v,P(u_{\epsilon})w).$$

Define now a linear map $T_{\epsilon}\colon V_{\delta}\to [W^{1,2}\cap L^{\infty}](M,\mathbb{R}^L)$ by setting
$$(T_{\epsilon}v)(x)=v_{\epsilon}(x):=P(u_{\epsilon}(x))v(x)$$
(and $v_{\epsilon}\equiv 0$ on $B_{\delta^2}(\Sigma)$). Since $v(x)\in T_{u(x)}N$ for a.e. $x\in M$, it's easy to see that
$$v_{\epsilon}\to v \text{ in }[W^{1,2}\cap L^{\infty}](M,\mathbb{R}^L)$$
as $\epsilon\to 0$, and the space 
$$V_{\delta,\epsilon}=T_{\epsilon}(V_{\delta})$$
has dimension 
$$\dim(V_{\delta,\epsilon})=\dim(V_{\delta})=\dim(V)$$
for $\epsilon$ sufficiently small. 

For every $v_{\epsilon}=T_{\epsilon}v$ in $V_{\delta,\epsilon}$, we clearly have $P^{\perp}(u_{\epsilon})v_{\epsilon}=0$, so by \eqref{dw.comp} and \eqref{hess.w.exp}, 
$$|D^2W(u_{\epsilon})(v_{\epsilon},v_{\epsilon})+\langle DW(u_{\epsilon}),B(u_{\epsilon})(v_{\epsilon},v_{\epsilon})\rangle|\leq Cd_N(u)^2|v_{\epsilon}|^2.$$
In particular, it follows that
\begin{equation*}
\begin{split}
&E_{\epsilon}''(u_{\epsilon})(v_{\epsilon},v_{\epsilon})=\int_{\Omega}|dv_{\epsilon}|^2+\epsilon^{-2} D^2W(u)(v_{\epsilon},v_{\epsilon})\,dv_g\\
&\leq \int_{\Omega}\left(|dv_{\epsilon}|^2-\langle \epsilon^{-2}DW(u_{\epsilon}),B(u_{\epsilon})(v_{\epsilon},v_{\epsilon})\rangle\right)\,dv_g
+\int_{\Omega}C\frac{d_N(u_{\epsilon})^2}{\epsilon^2}|v_{\epsilon}|^2\,dv_g\\
&\leq\int_{\Omega}\left(|dv_{\epsilon}|^2+\langle \Delta u_{\epsilon},B(u_{\epsilon})(v_{\epsilon},v_{\epsilon})\rangle\right)\,dv_g
+C\|v_\epsilon\|_{L^{\infty}}^2\int_{\Omega}\frac{d_N(u_{\epsilon})^2}{\epsilon^2}\,dv_g,
\end{split}
\end{equation*}
where in the last inequality we have used the fact that $u_{\epsilon}$ solves
$$\Delta u_{\epsilon}+\epsilon^{-2}DW(u_{\epsilon})=0.$$
By the smooth convergence $u_{\epsilon}\to u$ on $\Omega$, the $W^{1,2}\cap L^{\infty}$ convergence $v_{\epsilon}\to v$ and harmonicity of $u$, it follows that
\begin{equation*}
\begin{split}
&\lim_{\epsilon\to 0}E_{\epsilon}''(u_{\epsilon})(v_{\epsilon},v_{\epsilon})\leq \\
&\leq\int_{\Omega}(|dv|^2+\langle \Delta u,\II_N(u)(v,v)\rangle)\,dv_g
+C\|v\|_{L^{\infty}}^2\lim_{\epsilon\to 0}\int_{\Omega}\frac{d_N(u_{\epsilon})^2}{\epsilon^2}\,dv_g\\
&=\int_{\Omega}(|dv|^2-\langle \II_N(u)(du(e_i),du(e_i)),\II_N(u)(v,v)\rangle)\,dv_g \\
&+C\|v\|_{L^{\infty}}^2\lim_{\epsilon\to 0}\int_{\Omega}\frac{d_N(u_{\epsilon})^2}{\epsilon^2}\,dv_g\\
&=E''(u)(v,v)+C\|v\|_{L^{\infty}}^2\lim_{\epsilon\to 0}\int_{\Omega}\frac{d_N(u_{\epsilon})^2}{\epsilon^2}\,dv_g\\
&\leq -\frac{\beta}{2}\left(\|v\|_{L^{\infty}}^2+\|v\|_{W^{1,2}}^2\right)+C\|v\|_{L^{\infty}}^2\lim_{\epsilon\to 0}\int_{\Omega}\frac{d_N(u_{\epsilon})^2}{\epsilon^2}\,dv_g
\end{split}
\end{equation*}
Finally, observe that
$$2\epsilon^{-2}d_N(u_{\epsilon})=|\epsilon^{-2}DW(u_{\epsilon})|=|\Delta u_{\epsilon}|$$
on $\Omega$, and since $u_{\epsilon}\to u$ in $C^{\infty}(\Omega)$, it follows that
$$\lim_{\epsilon\to 0}\int_{\Omega}\frac{d_N(u_{\epsilon})^2}{\epsilon^2}=\lim_{\epsilon\to 0}\int_{\Omega}\frac{1}{2}d_N(u_{\epsilon})|\Delta u_{\epsilon}|=0.$$
Thus, for $\epsilon>0$ sufficiently small, we must have
$$E_{\epsilon}''(u)(v_{\epsilon},v_{\epsilon})\leq -\frac{\beta}{4}\left(\|v\|_{L^{\infty}}^2+\|v\|_{W^{1,2}}^2\right)$$
for all $v\in V_{\delta}.$

In other words, for $\epsilon$ sufficiently small, $E_{\epsilon}''(u_{\epsilon})$ is negative definite on the space $V_{\delta,\epsilon}$, and therefore
$$\dim V=\dim V_{\delta}=\dim V_{\delta,\epsilon}\leq \ind_{E_{\epsilon}}(u_{\epsilon}),$$
from which the desired index bound
$$\ind_E(u)\leq\liminf_{\epsilon\to 0}\ind_{E_{\epsilon}}(u_{\epsilon})$$
follows.
\end{proof}

\begin{remark}\label{loc.ind.lim} Like the analysis of \cite{LW99}, Lemma \ref{ind.lim} can easily be localized. In particular, let $\Omega^n$ be a compact manifold with boundary, with a smoothly converging family of metrics $g_{\epsilon}\to g$, and consider a family of critical points $u_{\epsilon}\colon(\Omega,g_\epsilon)\to \mathbb{R}^L$ for $E_{\epsilon,g_{\epsilon}}$ with $E_{\epsilon,g_{\epsilon}}(u_{\epsilon})\leq C$ and $\ind_{E_{\epsilon,g_{\epsilon}}}(u_{\epsilon})\leq I_0$ (where the index is defined with respect to compactly supported variations in $\Omega$). Then the maps $u_{\epsilon}$ converge weakly in $W^{1,2}$ to a weakly harmonic map $u\in W^{1,2}(\Omega,N)$ of Morse index $\ind_E(u)\leq I_0$ with respect to compactly supported variations in $\Omega$.

\end{remark}

The final ingredient needed to prove Theorem \ref{stat.lim.thm} is the following observation, showing that the harmonic maps $\mathbb{S}^2\to N$ arising along the energy concentration set as in Remark \ref{bubble.rk} must be stable if the maps $u_{\epsilon}$ have (locally) bounded index as critical points of $E_{\epsilon}$. This should be compared with the main observation of \cite{Hsu05}, where a similar conclusion is reached in the setting of stable stationary harmonic maps.

\begin{lemma}\label{bub.stab}
Let $\phi\colon \mathbb{R}^2\to N$ be a finite-energy harmonic map such that, for some $n\geq 3$, the harmonic map $u\colon\mathbb{R}^n\to N$ given by $u(x_1,\ldots,x_n)=\phi(x_1,x_2)$ has finite Morse index. Then $\phi\colon \mathbb{R}^2\to N$ is stable, and in particular can be identified with a stable harmonic map $\phi\colon \mathbb{S}^2\to N$.
\end{lemma}

\begin{proof} Identifying $\mathbb{R}^2$ with $\mathbb{S}^2\setminus \{p\}$ via stereographic projection, suppose, to the contrary, that the harmonic map $\phi\colon \mathbb{S}^2\to N$ is unstable. Then there exists a smooth, nonzero vector field $v\in C^{\infty}(S^2,\mathbb{R}^L)$ with $v(x)\in T_{\phi(x)}N$ for which $E''(\phi)(v,v)<0$. As in the proof of Lemma \ref{ind.lim}, we can use a logarithmic cutoff function (with $p$ playing the role of $\Sigma$) to perturb $v$ to a new field $\tilde{v}\in C_c^{\infty}(\mathbb{S}^2\setminus \{p\},\mathbb{R}^L)$ for which the second variation $E''(\phi)(\tilde{v},\tilde{v})$ remains negative.

In particular, it follows that $\phi\colon \mathbb{R}^2\to N$ is unstable with respect to compactly supported perturbations, so in what follows we may select some $v\in C_c^{\infty}(\mathbb{R}^2,\mathbb{R}^L)$ with $v(x)\in T_{\phi(x)}N$ such that
$$
E''(\phi)(v,v)<-\beta\int_{\mathbb{R}^2}|v|^2\,dx
$$
for some $\beta>0$.

Now, for $n\geq 3$, let $u\colon \mathbb{R}^n\to N$ be the harmonic map given by 
$$u(x_1,\ldots,x_n)=\phi(x_1,x_2).$$
For any $\psi\in C_c^{\infty}(\mathbb{R}^{n-2})$, we obtain a field $v_{\psi}\in C_c^{\infty}(\mathbb{R}^n,\mathbb{R}^L)$ with $v_{\psi}(x)\in T_{u(x)}N$ by setting
$$v_{\psi}(x)=\psi(x_3,\ldots,x_n)v(x_1,x_2).$$
Direct computation then gives
\begin{equation*}
\begin{split}
&E''(u)(v_{\psi},v_{\psi}) =\\
=&\int_{\mathbb{R}^n}|d(\psi v)|^2-\sum_{i=1}^2\psi^2\langle \II_N(\phi)(d\phi(e_i),d\phi(e_i)),\II_N(\phi)(v,v)\rangle\,dx\\
=&\left(\int_{\mathbb{R}^{n-2}}\psi^2\right)E''(\phi)(v,v)+\left(\int_{\mathbb{R}^2}|v|^2\right)\int_{\mathbb{R}^{n-2}}|d\psi|^2\,dx\\
<&-\beta \|v\|_{L^2(\mathbb{R}^2)}^2\|\psi\|_{L^2(\mathbb{R}^{n-2})}^2+\|v\|_{L^2(\mathbb{R}^2)}^2\|d\psi\|_{L^2(\mathbb{R}^{n-2})}^2.
\end{split}
\end{equation*}
Now, on $\mathbb{R}^{n-2}$, one can easily find a nonzero compactly supported function $\psi\in C_c^{\infty}(\mathbb{R}^{n-2})$ with $\frac{\|d\psi\|_{L^2}^2}{\|\psi\|_{L^2}^2}$ arbitrarily small, for instance by precomposing any given function in $C_c^{\infty}$ with a suitable dilation of $\mathbb{R}^{n-2}$. In particular, we can find a function $\psi_0\in C_c^{\infty}(\mathbb{R}^{n-2})$ with 
\begin{equation}\label{ray.small}
\|d\psi_0\|_{L^2}^2<\frac{\beta}{2}\|\psi_0\|_{L^2}^2,
\end{equation}
so that the computation above gives
$$E''(u)(v_{\psi_0},v_{\psi_0})<-\frac{\beta}{2}\|\psi_0\|_{L^2}^2\|v\|_{L^2}^2<0.$$
Choosing $R>0$ such that $\psi_0\in C_c^{\infty}(B_R(0))$, and setting
$$\mathcal{C}:=\mathrm{Span}(\{\psi_0(x+2R a)\mid a\in \mathbb{Z}^{n-2}\})\subset C_c^{\infty}(\mathbb{R}^{n-2}),$$
we see that $\mathcal{C}$ gives an infinite dimensional space of test functions satisfying \eqref{ray.small}, and consequently
$$\{v_{\psi}\mid \psi \in \mathcal{C}\}$$
gives an infinite dimensional space of variation vector fields on which $E''(u)$ is negative definite. In particular, $\ind_E(u)=\infty$.
\end{proof}

We now have all the ingredients we need to complete the proof of Theorem \ref{stat.lim.thm}.

\begin{proof}[Proof of Theorem \ref{stat.lim.thm}]

Let $u_{\epsilon}\colon M\to \mathbb{R}^L$ be a family of critical points for $E_{\epsilon}$ with 
$$E_{\epsilon}(u_{\epsilon})\leq C\text{ and }\ind_{E_{\epsilon}}(u_{\epsilon})\leq I_0.$$
By Theorem \ref{lw.thm}, we can pass to a subsequence to obtain a weakly harmonic map $u\in W^{1,2}(M,N)$ such that
$$u_{\epsilon}\to u\text{ weakly in }W^{1,2},$$
and the failure of strong convergence is captured by the discrepancy measure
$$\nu=\lim_{\epsilon\to 0}e_{\epsilon}(u_{\epsilon})dv_g-\frac{1}{2}|du|_g^2dv_g.$$

Once we've shown that $\nu=0$, strong convergence will follow, as will the stationarity of the limit map $u\colon M\to N$, since in that case it follows that
\begin{eqnarray*}
\int \frac{1}{2}|du|^2\mathrm{div}(X)-\langle du^*du,DX\rangle&=&\lim_{\epsilon \to 0}\int e_{\epsilon}(u_{\epsilon})\mathrm{div}(X)-\langle du_{\epsilon}^*du_{\epsilon},DX\rangle\\
&=&0
\end{eqnarray*}
for any smooth vector field $X\in \Gamma(TM)$--using the fact that every critical point $u_{\epsilon}$ of $E_{\epsilon}$ is smooth and, thus, automatically satisfies
$$
\mathrm{div}(e_{\epsilon}(u)g-du_{\epsilon}^*du_{\epsilon})=0.
$$
The index bound $\ind_E(u)\leq I_0$ follows immediately from Lemma \ref{ind.lim}. Thus, to complete the proof, all that remains is to show that $\nu\equiv 0$. 

Suppose, to obtain a contradiction, that $\nu\neq 0$. By Remark \ref{bubble.rk}, it follows that there exists a sequence of points $x_\epsilon\to x$ and a sequence of scales $r_{\epsilon}>0$ with $\frac{\epsilon}{r_{\epsilon}}\to 0$ such that the rescaled maps
$$\tilde{u}_{\epsilon}(y)=u_{\epsilon}(\exp_{x_\epsilon}(y/r_{\epsilon}))$$
converge in $C^{\infty}_{loc}(\mathbb{R}^n)$ to a harmonic map $u\colon \mathbb{R}^n\to N$ of the form (up to rotations)
$$\tilde{u}(x_1,\ldots,x_n)=\phi(x_1,x_2),$$
where $\phi\colon \mathbb{R}^2\to N$ is a smooth, nonconstant, finite-energy harmonic map. 

For any compact subset $\Omega\subset \mathbb{R}^n$ and $\epsilon>0$ sufficiently small, we see that $\tilde{u}_{\epsilon}$ defines a critical point for $E_{r_{\epsilon}^{-1}\epsilon,\tilde{g}_{\epsilon}}$ with respect to the metric $\tilde{g}_{\epsilon}:=\Phi_{\epsilon}^*g$ obtained by pulling back $g$ via $\Phi_{\epsilon}(y)=\exp_{x_\eps}(y/r_{\epsilon})$. Moreover, it is easy to check that for $\epsilon>0$ sufficiently small,
$$\ind_{E_{r_{\epsilon}^{-1}\epsilon,\tilde{g}_{\epsilon}}}(\tilde{u}_{\epsilon};\Omega)\leq \ind_{E_{\epsilon}}(u_{\epsilon})\leq I_0,$$
so by Lemma \ref{ind.lim} (and Remark \ref{loc.ind.lim}), it follows that the harmonic map $\tilde{u}$ has
$$\ind_E(\tilde{u};\Omega)\leq I_0.$$

In particular, since this holds for any compact subset $\Omega\subset \mathbb{R}^n$, we have
$$\ind_E(\tilde{u})\leq I_0$$
with respect to all compactly supported variations on $\mathbb{R}^n$. But then Lemma \ref{bub.stab} implies that the harmonic map $\phi\colon \mathbb{R}^2\to N$ can be identified with a nonconstant \emph{stable} harmonic map $\phi\colon \mathbb{S}^2\to N$, which cannot exist by the key hypothesis for $N$. Thus, we arrive at a contradiction, and deduce that $\nu\equiv 0$, as desired.

\end{proof}

\subsection{Partial regularity for the min-max harmonic maps}

For locally energy-minimizing harmonic maps $u\colon M\to N$, the foundational results of Schoen and Uhlenbeck \cite{SU82} show that the singular set has dimension $\dim(\Sing(u))\leq n-3$, and this can be improved to $\dim(\Sing(u))\leq n-m$ in cases where $N$ admits no $0$-homogeneous energy-minimizing maps $\mathbb{R}^{m-1}\to N$ for some $m\geq 4$. 

Without an energy-minimizing assumption, the results of \cite{Bet93} imply that a stationary harmonic map $u\colon M^n\to N$ is in general smooth away from a singular set $\Sing(u)$ with vanishing $(n-2)$-dimensional Hausdorff measure $\mathcal{H}^{n-2}(\Sing(u))=0$, but it remains a challenging open problem to improve on this estimate (most optimistically, showing that $\dim(\Sing(u))\leq n-3$ in general). The key difficulty in extending Schoen and Uhlenbeck's regularity theory to general stationary harmonic maps is the noncompactness of stationary harmonic maps in $W^{1,2}$; in particular, the sequence of harmonic maps obtained by rescaling a given harmonic map at a singular point may not converge strongly to a homogeneous ``tangent map", making it impossible a priori to carry out Schoen and Uhlenbeck's dimension reduction argument.

Fortunately, the harmonic maps constructed in Theorem \ref{stat.lim.thm} have more structure than general stationary harmonic maps. Namely, they are harmonic maps of \emph{finite Morse index} taking values in a target $N$ admitting \emph{no stable harmonic $2$-spheres}. Indeed, since $N$ admits no stable harmonic $2$-sphere, we can use ideas of Hsu \cite{Hsu05}, showing that locally stable harmonic maps to $N$ satisfy a partial regularity theory comparable to that of energy-minimizers. 

\begin{theorem}[Hsu~\cite{Hsu05}]\label{loc.stab.reg}
Let $u\colon M\to N$ be a locally stable stationary harmonic map to a compact Riemannian manifold $N$ for which every stable harmonic map $\mathbb{S}^2\to N$ is constant. Then $u$ is smooth away from a singular set $\Sing(u)$ of dimension $\dim(\Sing(u))\leq n-3$. Moreover, $\dim(Sing(u))\leq n-m$ if $N$ admits no stable $0$-homogeneous harmonic map $v: \mathbb{R}^{m-1}\to N$ for some $m\geq 4$.
\end{theorem}

The proof of Theorem \ref{loc.stab.reg} rests primarily on the observation that the nonexistence of stable harmonic maps $\mathbb{S}^2\to N$ rules out energy concentration for sequences of \emph{stable} harmonic maps to $N$ \cite[Lemma 2.2]{Hsu05}, by the same mechanism exploited in the proof of Theorem \ref{stat.lim.thm} above. In particular, the sequences of harmonic maps obtained by rescaling around a singular point converge \emph{strongly} in $W_{loc}^{1,2}$ to a homogeneous tangent map, allowing one to apply the same dimension reduction techniques used by Schoen--Uhlenbeck in the energy-minimizing case.

We observe next that stationary harmonic maps of finite Morse index are locally stable near every point, allowing us to apply Theorem \ref{loc.stab.reg} to deduce the regularity of these maps. 

\begin{lemma}\label{loc.stab.lem} Let $u\colon M\to N$ be a stationary harmonic map of finite Morse index. Then for every $p\in M$, there exists some $r_p>0$ such that $u$ is stable in the ball $B_{r_p}(p)$.
\end{lemma}
\begin{proof} We begin with a standard trick (cf. e.g. \cite[Lemma 3.9]{Gu}), showing that for every point $x\in M$ there exists a radius $r(x)>0$ such that $u$ is stable on any annulus $A_{s,t}(p):=B_t(p)\setminus B_s(p)$ with $0<s<t\leq r(x)$. Indeed, if this were not the case, then we could find an infinite sequence of radii $t_1>s_1>t_2>s_2>\cdots$ such that $u$ is unstable on the annuli $A_{s_i,t_i}(p)$; but since the annuli $A_{s_i,t_i}(p)$ are pairwise disjoint, this would violate the assumption that $\ind_E(u)<\infty$. 

Now, with $r=r(x)>0$ as above, we wish to show that $u$ is stable on $B_r(x)$--i.e., the second variation $E''(u)$ is nonnegative definite on the space
$$\mathcal{V}(u;B_{r}(x)):=\{v\in [W_0^{1,2}\cap L^{\infty}](B_r(x))\mid v(x)\in T_{du(x)}N\text{ a.e. }x\in M\}.$$
To this end, let $v\in \mathcal{V}(u;B_r(x))$, and for $\delta\in (0,r(x))$, let $\varphi_{\delta}\in C^{\infty}(M)$ be a cutoff function satisfying 
$$\varphi_{\delta}\equiv 1\text{ on }M\setminus B_{2\delta}(x),\text{ }\varphi\equiv 0\text{ on }B_{\delta}(x),\text{ and }|d\varphi_{\delta}|\leq \frac{C}{\delta},$$
so that
$$\|d\varphi_{\delta}\|_{L^2}^2\leq C\delta^{n-2}.$$
Since $\varphi_{\delta}v$ is supported on the annulus $A_{\delta, r(x)}(p)$, on which $u$ is stable by our choice of $r(x)$, it follows that
$$E''(u)(\varphi_{\delta}v,\varphi_{\delta}v)\geq 0.$$
On the other hand, note that
\begin{equation*}
\begin{split}
&E''(u)(v)-E''(u)(\varphi_{\delta}v)=\int_{B_r(x)} (|dv|^2-|d(\varphi_{\delta}v)|^2)\,dv_g \\
-&\int_{B_r(x)}(1-\varphi_{\delta}^2)\langle \II_N(u)(du(e_i),du(e_i)),\II_N(u)(v,v)\rangle\,dv_g \\
=&\int_{B_r(x)} (1-\varphi_{\delta}^2)|dv|^2-2\langle v\otimes d\varphi_{\delta},\varphi_{\delta}dv\rangle-|v|^2|d\varphi_{\delta}|^2\,dv_g\\
-&\int_{B_r(x)}(1-\varphi_{\delta}^2)\langle \II_N(u)(du(e_i),du(e_i)),\II_N(u)(v,v)\rangle\,dv_g\\
\geq& -2\|v\|_{L^{\infty}}\|dv\|_{L^2}\|d\varphi_{\delta}\|_{L^2}-\|v\|_{L^{\infty}}^2\|d\varphi_{\delta}\|_{L^2}^2\\
-&\|\II_N\|_{L^{\infty}(N)}^2\|v\|_{L^{\infty}}^2\|du\|_{L^2(B_{2\delta}(p))}^2.
\end{split}
\end{equation*}
Since $v\in W^{1,2}\cap L^{\infty}$ and $\|d\varphi_{\delta}\|_{L^2}^2\leq C\delta^{n-2}$, it is easy see that the final lower bound in the preceding string of inequalities vanishes as $\delta\to 0$, and since $E''(u)(\varphi_{\delta})\geq 0$, it follows that
$$E''(u)(v)\geq 0\text{ for all }v\in \mathcal{V}(u;B_r(x)),$$
as desired.

\end{proof}

Combining Lemma \ref{loc.stab.lem} with Theorem \ref{loc.stab.reg}, we arrive immediately at the following sharp partial regularity result for stationary harmonic maps to $N$ of finite Morse index. 

\begin{theorem}\label{ind.bd.reg}
Let $u\colon M\to N$ be a stationary harmonic map of finite Morse index to a compact Riemannian manifold $N$ for which every stable harmonic map $\mathbb{S}^2\to N$ is constant. Then $u$ is smooth away from a singular set $\Sing(u)$ of dimension $\dim(\Sing(u))\leq n-3$. Moreover, $\dim(\Sing(u))\leq n-k$ if $N$ admits no stable $0$-homogeneous harmonic map $v\colon \mathbb{R}^{k-1}\to N$ for some $k\geq 4$.
\end{theorem}

\begin{remark} In \cite{HsuLi}, compactness and regularity results are given for harmonic maps $u\colon M\to N$ with bounded Morse index $\ind_E(u)\leq I_0$ when the target $N$ is assumed to admit no nonconstant harmonic map $\phi\colon\mathbb{S}^2\to N$ \emph{with matching index bound} $\ind_E(\phi)\leq I_0$. The simple observation that--by virtue of Lemma \ref{bub.stab}--the hypothesis on $N$ can be weakened to the absence of \emph{stable} harmonic $2$-spheres is critical for applications to variational constructions like the one considered here, allowing us to obtain compactness and regularity results for harmonic maps obtained from multi-parameter min-max constructions. 

\end{remark}

Finally, by combining Proposition \ref{pert.minmax}, Theorem \ref{stat.lim.thm}, and Theorem \ref{ind.bd.reg}, we arrive at the following existence result, from which Theorem \ref{exthm} clearly follows.

\begin{theorem}\label{sec2.exthm} In the setting of Section \ref{pert.minmax.sec}, if there exist no nonconstant stable harmonic maps $\mathbb{S}^2\to N$, then there exists a stationary harmonic map $u\colon M^n\to N$ of Morse index $\ind_E(u)\leq \ell+1$, with 
$$E(u)=\mathcal{E}_f(M,g)>0.$$
Moreover, $u$ is smooth away from a (possibly empty) singular set $\Sing(u)$ of dimension
$$\dim(\Sing(u))\leq n-m\leq n-3,$$
where $m$ is the smallest dimension for which there exists a nonconstant stable stationary $0$-homogeneous map $\phi\colon \mathbb{R}^m\to N$.
\end{theorem}

\section{Improved regularity and stabilization in the sphere-valued case}

\label{sec:sphere}

For any $k\geq 3$, we can take $N=\mathbb{S}^k$, $\ell=k$ and $f=\mathrm{Id}\colon \mathbb{S}^k\to \mathbb{S}^k$ in Theorem \ref{sec2.exthm} to deduce the existence of a nonconstant stationary sphere-valued harmonic map
$$u_k\colon M^n\to \mathbb{S}^k$$
of Morse index $\ind_E(u_k)\leq k+1$, from any compact manifold $(M^n,g)$ of dimension $n\geq 3$. For energy-minimizing maps to spheres $\mathbb{S}^k$ of dimension $k\geq 3$, Schoen and Uhlenbeck obtained a refined partial regularity theory \cite{SU84}, showing in particular that such maps have singular set $\Sing(u)$ of Hausdorff dimension $\dim(\Sing(u))\leq n-7$ for $k$ sufficiently large. Their results were later extended to the case of \emph{stable} harmonic maps by Hong-Wang \cite{HW99} and further refined by Lin and Wang \cite{LW06}, yielding the following partial regularity result.
\begin{theorem}(Lin, Wang \cite{LW06}\label{sph.reg.thm})
Let $u\colon M^n\to \mathbb{S}^k$ be a locally stable stationary harmonic map. Then $u$ is smooth away from a singular set $\Sing(u)$ of Hausdorff dimension 
$$\dim (\Sing(u))\leq n-k-1\text{\hspace{3mm} for\hspace{3mm} }3\leq k\leq 5,$$
$$\dim(\Sing(u))\leq n-6\text{\hspace{3mm} for\hspace{3mm} } 6\leq k\leq 9,$$
and
$$\dim(\Sing(u))\leq n-7\text{\hspace{3mm} for\hspace{3mm} }k\geq 10.$$
\end{theorem}
Since Lemma \ref{loc.stab.lem} shows that stationary harmonic maps of finite Morse index are locally stable near each point, it follows that the same regularity result holds in the finite Morse index case. In particular, combining this with the existence result obtained by applying Theorem~\ref{sec2.exthm} with $\ell=k$ and $N=\mathbb{S}^k$, we obtain the following.
\begin{theorem}\label{sk.minmax.thm} On any compact Riemannian manifold $(M^n,g)$ of dimension $n\geq 3$, and for any $k\geq 3$, there exists a nonconstant stationary harmonic map $u_k\colon M\to \mathbb{S}^k$ of Morse index $\ind_E(u_k)\leq k+1$, smooth away from a singular set $Sing(u)$ of dimension
$$\dim (\Sing(u_k))\leq n-k-1\text{\hspace{3mm} for\hspace{3mm} }3\leq k\leq 5,$$
$$\dim(\Sing(u_k))\leq n-6\text{\hspace{3mm} for\hspace{3mm} }6\leq k\leq 9,$$
and
$$\dim(\Sing(u_k))\leq n-7\text{\hspace{3mm} for\hspace{3mm} }k\geq 10.$$
\end{theorem}

Note that in the low-dimensional cases $n=\dim(M)=3,4,5$, it follows in particular that these min-max harmonic maps
$$u_k\colon M^n\to S^k$$
are \emph{smooth} as soon as $k\geq n$. In the remainder of this section, we upgrade this qualitative regularity statement to a priori gradient estimates independent of $k$, which allow us to show that the maps $u_k$ stabilize as $k$ becomes large--in the sense that the image $u_k(M)$ lies inside some equatorial subsphere $\mathbb{S}^{k_0}\subset \mathbb{S}^k$ of dimension $k_0(M,g)$ fixed independent of $k$. As in \cite{KS1}, this stabilization result is the key analytic ingredient that allows us to relate these min-max harmonic maps to eigenvalue optimization problems, as described in the next section.

\subsection{A priori gradient estimates for low-index harmonic maps to spheres}

Let $u\colon M^n\to \mathbb{S}^k$ be a smooth harmonic map, and let $\Omega\subset M$ be a domain. As in \cite{SU84, HW99, LW06}, our estimates begin with the observation that, when the index $\ind_E(u;\Omega)$ is small, testing the second variation of energy $E''(u)$ against variations of the form
$$v_a:=\psi\cdot(a-\langle u,a\rangle u)\in \Gamma(u^*(T\mathbb{S}^k))$$
for $a\in \mathbb{S}^k$ gives rise to useful estimates in terms of scalar test functions $\psi\in C_c^{\infty}(\Omega)$. In particular, recalling that
$$E''(u)(v,v)=\int_M|dv|^2-|du|^2|v|^2\,dv_g,$$
and computing
$$|dv_a|^2=\psi^2(u_a^2|du|^2+|du_a|^2)+(1-u_a^2)|d\psi|^2-2\psi\langle d\psi, u_adu_a\rangle$$
and
$$|v_a|^2=\psi^2(1-u_a^2),$$
where we've set $u_a:=\langle u,a\rangle$, it follows that 
\begin{equation*}
\begin{split}
&E''(u)(v_a,v_a)=\\
=&\int_M\psi^2([2u_a^2-1]|du|^2+|du_a|^2)+(1-u_a^2)|d\psi|^2-2\psi\langle d\psi, u_adu_a\rangle\,dv_g\\
=&\int_M \psi^2(3u_a^2-1)|du|^2+(1-u_a^2)|d\psi|^2-4\psi\langle d\psi, u_adu_a\rangle\,dv_g.
\end{split}
\end{equation*}

Now, fixing an arbitrary $\psi\in C_c^{\infty}(\Omega)$, let $a_1,\ldots,a_{k+1}\in \mathbb{S}^k$ be an orthonormal basis for $\mathbb{R}^{k+1}$ diagonalizing the quadratic form 
$$a\mapsto E''(u)(v_a,v_a),$$
and write $v_j=v_{a_j}$. Summing the preceding identity for $E''(u)(v_j,v_j)$ over $j=1,\ldots,k+1$, we see that
\begin{equation}\label{trace.hess.comp}
\sum_{j=1}^{k+1}E''(u)(v_j,v_j)=\int_M \psi^2(2-k)|du|^2+k|d\psi|^2\,dv_g.
\end{equation}
On the other hand, if $\ind_E(u;\Omega)\leq m<k-2$, then we must have $E''(u)(v_i,v_i)\geq 0$ for $i>m$, and so--writing $u_j=u_{a_j}$--we see that
\begin{equation*}
\begin{split}
&\sum_{j=1}^{k+1}E''(u)(v_j,v_j)\geq \sum_{j=1}^mE''(u)(v_j,v_j)\\
=&\sum_{j=1}^m\int_M \psi^2(3u_j^2-1)|du|^2+(1-u_j^2)|d\psi|^2-4\psi\langle d\psi, u_jdu_j\rangle\,dv_g\\
=&\sum_{j=1}^m\int_M \psi^2(3u_j^2-1)|du|^2+(1-u_j^2)|d\psi|^2-\langle d(\psi^2), d(u_j^2)\rangle\,dv_g\\
=&\sum_{j=1}^m\int_M \psi^2(3u_j^2-1)|du|^2+(1-u_j^2)|d\psi|^2+\psi^2(|du_j|^2-u_j^2|du|^2)\,dv_g\\
\geq& -m\int_M |du|^2\psi^2\,dv_g+(m-1)\int_M |d\psi|^2\,dv_g.
\end{split}
\end{equation*}
Combining this with \eqref{trace.hess.comp}, it then follows that
$$(k-2-m)\int_M |du|^2\psi^2\,dv_g\leq (k-m+1)\int_M |d\psi|^2\,dv_g.$$
In other words, we've established the following.
\begin{lemma}\label{low.ind.lem} Let $u\colon M^n\to \mathbb{S}^k$ be a smooth harmonic map, and let $\Omega \subset M$ be a domain in $M$. If
$$\ind_E(u;\Omega)=k-2-p<k-2,$$
then
\begin{equation}\label{grad.control}
\frac{p}{p+3}\int_M |du|^2\psi^2\,dv_g\leq \int_M |d\psi|^2\,dv_g
\end{equation}
for every $\psi\in W_0^{1,2}(\Omega).$
\end{lemma}

Next, we argue along the lines of \cite{SU84} and \cite{LW06}, combining Lemma \ref{low.ind.lem} with the Bochner identity for $\mathbb{S}^k$-valued harmonic maps, testing \eqref{grad.control} against certain powers $\psi=|du|^{\alpha}$ of the energy density. For a harmonic map $u\colon M\to \mathbb{S}^k$, recall that the Bochner identity gives
\begin{eqnarray*}
-\frac{1}{2}\Delta |du|^2&=&|\Hess(u)|^2+\langle \Ric,du^*du\rangle-|du|^4\\
&=&|\Hess(u)^T|^2+\sum_{i,j=1}^n\langle du(e_i),du(e_j)\rangle^2+\langle \Ric_M, du^*du\rangle-|du|^4,
\end{eqnarray*}
where in the second inequality we decomposed the $\mathbb{R}^{k+1}$-valued Hessian
$$\Hess(u)=\Hess(u)^T-u\otimes du^*du$$
into components tangential and normal to $T\mathbb{S}^k$. It is easy to check that
$$\sum_{i,j=1}^n\langle du(e_i),du(e_j)\rangle^2\geq \frac{1}{n}|du|^4,$$
and \cite[Proposition 2.3]{LW06} (originally stated for $M=\mathbb{S}^n$, but the domain geometry plays no role in the proof) gives us the refined Kato inequality
$$|\Hess(u)^T|^2\geq \frac{n}{n-1}|d|du||^2.$$
Similar to \cite{LW06}, we then combine these estimates with the Bochner identity to deduce that
\begin{equation}\label{bochner}
-\frac{1}{2}\Delta |du|^2\geq \frac{n}{n-1}|d|du||^2-\|\Ric_M\|_{L^{\infty}}|du|^2-\frac{n-1}{n}|du|^4
\end{equation}
for any smooth harmonic map $u\colon M^n\to \mathbb{S}^k$. And as a consequence, for any $\alpha\geq 2$, we have
\begin{equation*}
\begin{split}
&-\frac{1}{\alpha}\Delta |du|^{\alpha}\geq\\
\geq&\left(\frac{n}{n-1}+\alpha-2\right)|du|^{\alpha-2}|d|du||^2
-\|\Ric_M\|_{L^{\infty}}|du|^{\alpha}-\frac{n-1}{n}|du|^{2+\alpha}\\
=&\frac{4}{\alpha^2}\left(\frac{n}{n-1}+\alpha-2\right)|d|du|^{\alpha/2}|^2
-\|\Ric_M\|_{L^{\infty}}|du|^{\alpha}-\frac{n-1}{n}|du|^{2+\alpha}.
\end{split}
\end{equation*}
Setting
$$Q(n,\alpha)^{-1}:=\frac{4}{\alpha^2}\left(\frac{n}{n-1}+\alpha-2\right),$$
for any test function $\varphi\in C_c^{\infty}(\Omega)$, it then follows that
\begin{equation*}
\begin{split}
&\int_M |d(\varphi |du|^{\alpha/2})|^2\,dv_g=\int_M |du|^{\alpha}|d\varphi|^2+\frac{1}{2}\langle d(\varphi^2),d |du|^{\alpha}\rangle+\varphi^2|d|du|^{\alpha/2}|^2\,dv_g \\
&\leq \int_M|du|^{\alpha}|d\varphi|^2+\frac{1}{2}\langle d(\varphi^2),d |du|^{\alpha}\rangle\,dv_g +\\
&+Q(n,\alpha)\int_M \varphi^2\left(\frac{n-1}{n}|du|^{2+\alpha}+\|\Ric_M\|_{L^{\infty}}|du|^{\alpha}-\frac{1}{\alpha}\Delta |du|^{\alpha}\right)\,dv_g\\
&=\int_M |du|^{\alpha}\left(|d\varphi|^2+C(M,n,\alpha)\varphi^2+C(n,\alpha)\Delta (\varphi^2)\right)\,dv_g +\\
&+Q(n,\alpha)\frac{n-1}{n}\int_M \varphi^2|du|^{2+\alpha}\,dv_g.
\end{split}
\end{equation*}
Combining these computations with Lemma \ref{low.ind.lem}, taking $\psi=\varphi |du|^{\frac{\alpha}{2}}$ as the test function, we arrive at the following key lemma.

\begin{lemma}\label{reg.imp.lem}
Let $u\colon M^n\to \mathbb{S}^k$ be a smooth harmonic map, such that $u$ has Morse index
$$\ind_E(u;\Omega)\leq k-2-p<k-2$$
on some domain $\Omega\subset M$. Then for $\alpha\in [2,\infty)$ and any $\varphi\in C_c^{\infty}(\Omega)$, we have
$$\left(\frac{p}{p+3}-P(n,\alpha)\right)\int_\Omega \varphi^2|du|^{2+\alpha}\,dv_g\leq C_{\alpha}\||d\varphi|^2+|\varphi||\Delta\varphi|\|_{L^{\infty}}\int_{\Omega}|du|^{\alpha}\,dv_g,$$
where $C_{\alpha}$ is a constant depending on $\alpha$ and the geometry of $M$, and
$$P(n,\alpha):=Q(n,\alpha)\frac{n-1}{n}=\frac{(n-1)^2\alpha^2}{4n[(n-1)\alpha+2-n]}.$$
\end{lemma}

Of particular importance are the cases $\alpha=2$ and $\alpha=4$. It is straightforward to check that
$$P(n,2)=\frac{(n-1)^2}{n^2}\leq \frac{16}{25}\text{ for }n\leq 5$$
while
$$P(n,4)=\frac{4(n-1)^2}{n(3n-2)}\leq \frac{64}{65}\text{ for }n\leq 5.$$
With Lemma \ref{reg.imp.lem} in hand, we can now prove the following a priori $L^6$ gradient estimate for low-index harmonic maps $u\colon M^n\to \mathbb{S}^k$ with $n\leq 5$ and $k$ large.

\begin{proposition}\label{l6.prop} Let $u\colon M^n\to \mathbb{S}^k$ be a smooth harmonic map where $n\leq 5$, and let $B_{2r}(p)\subset M$ be a geodesic ball of radius $2r$. If $k>202$ and 
$$\ind_E(u;B_{2r}(p))\leq k-2-200,$$
then 
$$\|du\|^6_{L^6(B_{r/2}(p))}\leq C(M)r^{-4}\|du\|_{L^2(B_{2r}(p))}^2.$$
\end{proposition}
\begin{proof}
By assumption, the hypotheses of Lemma \ref{reg.imp.lem} hold with $\Omega=B_{2r}(p)$ and $p=200$, so that $\frac{p}{p+3}=\frac{200}{203}$, and the lemma gives
$$\left(\frac{200}{203}-P(n,\alpha)\right)\int_{B_{2r}(p)} \varphi^2|du|^{2+\alpha}\leq C_{\alpha}\||d\varphi|^2+|\varphi||\Delta\varphi|\|_{L^{\infty}}\int_{B_{2r}(p)}|du|^{\alpha}$$
for any $\varphi\in C_c^{\infty}(B_{2r}(p))$. Taking $\alpha=2$ and noting that
$$\frac{200}{203}-P(n,2)\geq \frac{200}{203}-\frac{16}{25}>0,$$
it follows that
$$\int_{B_{2r}(p)} \varphi^2 |du|^4\,dv_g\leq C\||d\varphi|^2+|\varphi||\Delta\varphi|\|_{L^{\infty}}\int_{B_{2r}(p)}|du|^2\,dv_g.$$
Choosing $\varphi\in C_c^{\infty}(B_{2r}(p))$ such that $\varphi\equiv 1$ on $B_r(p)$, $|d\varphi|\leq \frac{C}{r}$, and $|\Delta \varphi|\leq \frac{C}{r^2}$, we deduce in particular that
\begin{equation}\label{l4.gradest}
\int_{B_r(p)}|du|^4\,dv_g\leq C(M)r^{-2}\int_{B_{2r}(p)}|du|^2\,dv_g.
\end{equation}

Next, we apply Lemma \ref{reg.imp.lem} again, this time with $\Omega=B_r(p)$, $p=200$, and $\alpha=4$, and since (for $n\leq 5$)
$$\frac{p}{p+3}-P(n,4)=\frac{200}{203}-P(n,4)\geq \frac{200}{203}-\frac{64}{65}=\frac{8}{(203)\cdot 65}>0,$$
the lemma yields an estimate of the form
$$\int_{B_r(p)} \varphi^2 |du|^6\,dv_g \leq C\||d\varphi|^2+|\varphi||\Delta \varphi|\|_{L^{\infty}}\int_{B_r(p)}|du|^4\,dv_g$$
for any $\varphi\in C_c^{\infty}(B_r(p))$. As before, we choose $\varphi$ such that 
$$\varphi\equiv 1\text{ on }B_{r/2}(p),\text{ }|d\varphi|\leq \frac{C}{r},\text{ and }|\Delta \varphi|\leq \frac{C}{r^2},$$
obtaining the estimate
$$\int_{B_{r/2}(p)}|du|^6\,dv_g\leq \frac{C}{r^2}\int_{B_r(p)}|du|^4\,dv_g.$$
Combining this with \eqref{l4.gradest} then gives the desired estimate.

\end{proof}

\subsection{Stabilization for $\mathbb{S}^k$-valued harmonic maps of index $\ll2k$}

Armed with the local estimates of Proposition \ref{l6.prop}, we next establish \emph{global} $L^6$ estimates for the min-max harmonic maps $u_k\colon M^n\to \mathbb{S}^k$ given by Theorem \ref{sk.minmax.thm}, when $n=3,4,5$. More generally, the following arguments can be applied to any family of harmonic maps $u_k\colon M^n\to \mathbb{S}^k$ with $\limsup_{k\to\infty}\frac{\ind_E(u_k)}{k}<2$. 

First, we observe that a nonconstant harmonic map $u\colon M^n\to \mathbb{S}^k$, where $n\geq 3$, must have index $\geq k-2$ on the complement of any small ball.

\begin{lemma}\label{no.ind.conc} Let $n\geq 3$. There exists $\delta(M)>0$ such that for any nonconstant harmonic map $u\colon M^n\to \mathbb{S}^k$
and any geodesic ball $B_{\delta}(p)\subset M$ of radius $\delta$, 
$$\ind_E(u; M\setminus B_{\delta}(p))\geq k-2.$$
\end{lemma}
\begin{proof} Suppose, to the contrary, that $\ind_E(u; M\setminus B_{\delta}(p))\leq k-3$. Then by Lemma \ref{low.ind.lem}, it follows that
\begin{equation}\label{low.en.est}
\frac{1}{4}\int_M |du|^2\psi^2\,dv_g\leq \int_M |d\psi|^2\,dv_g
\end{equation}
for any $\psi\in W_0^{1,2}(M\setminus B_{\delta}(p))$. Choosing $\psi\in W_0^{1,2}(M\setminus B_{\delta}(p))$ such that $\psi \equiv 1$ on $M\setminus B_{2\delta}(p)$ and $|d\psi|\leq \frac{C}{\delta}$ on $B_{2\delta}(p)\setminus B_{\delta}(p),$ we see that
$$\int_M |d\psi|^2\,dv_g\leq C \delta^{n-2},$$
which together with \eqref{low.en.est} gives
\begin{equation}\label{tiny.en}
\int_{M\setminus B_{2\delta}(p)}|du|^2\,dv_g\leq C\delta^{n-2}.
\end{equation}

On the other hand, by the well-known energy-monotonicity properties of harmonic maps, we have an estimate of the form
$$\int_{B_{2\delta}(p)}|du|^2\,dv_g\leq C(M)\delta^{n-2}\int_M |du|^2\,dv_g.$$
Provided $\delta>0$ is small enough that $\delta^{n-2}C(M)<\frac{1}{2}$, this estimate can be combined with \eqref{tiny.en} to give
\begin{equation}\label{glob.en.bd}
\int_M|du|^2\,dv_g\leq C'(M)\delta^{n-2}.
\end{equation}
However, by Proposition \ref{skener.prop} in the appendix, there is a universal lower bound $\beta(M)>0$ independent of $k$ such that
$$E(u)=\frac{1}{2}\int_M|du|^2\geq \beta$$
for any nonconstant harmonic map $u\colon M\to \mathbb{S}^k$. Thus, taking
$$\delta(M):=(\beta/C'(M))^{\frac{1}{n-2}},$$
we get a contradiction with \eqref{glob.en.bd}, completing the proof.
\end{proof}


We can now combine Lemma \ref{no.ind.conc} with Proposition \ref{l6.prop} to arrive at global $L^6$ gradient estimates when $\ind_E(u_k)\leq k+1$ and $k$ is sufficiently large.

\begin{lemma}\label{glob.l6.bd} Let $3\leq n\leq 5$ and $k\geq 205$, and let $u\colon M\to \mathbb{S}^k$ be a harmonic map of index $\ind_E(u)\leq k+1$. Then there exists a constant $C(M)<\infty$ independent of $u$ and $k$ such that 
\begin{equation}
\|du\|_{L^6(M)}\leq C(M)
\end{equation}
and
\begin{equation}
\||du|^2\|_{W^{1,2}(M)}\leq C(M).
\end{equation}
\end{lemma}
\begin{proof} By Lemma \ref{no.ind.conc}, there exists $\delta(M)>0$ such that for every ball $B_{\delta}(p)\subset M$ of radius $\delta$, we have
$$\ind_E(u;M\setminus B_{\delta}(p))\geq k-2.$$
Since $M\setminus B_{\delta}(p)$ and $B_{\delta}(p)$ are disjoint, it follows in particular that
\begin{equation*}
\begin{split}
\ind_E(u;B_{\delta}(p))+k-2 & \leq \ind_E(u;B_{\delta}(p))+\ind_E(u;M\setminus B_{\delta}(p)) \\
&\leq\ind_E(u;M)\leq k+1,
\end{split}
\end{equation*}
so that
$$\ind_E(u;B_{\delta}(p))\leq 3.$$

Now, since $k\geq 205$, it follows that on every ball $B_{\delta}(p)\subset M$ of radius $\delta$, we have
$$\ind_E(u;B_{\delta}(p))\leq 3\leq k-2-200,$$
so the hypotheses of Proposition \ref{l6.prop} hold with $r=\delta/2$, and we can apply the proposition to deduce that
$$\|du\|_{L^6(B_{\delta/4}(p))}^6\leq C'(M)\delta^{-4}\|du\|_{L^2(B_{\delta}(p))}^2.$$
By a simple covering argument and H\"older's inequality, it follows that
$$\|du\|_{L^6(M)}^6\leq C(M)\|du\|_{L^2(M)}^2\leq C'(M)\|du\|_{L^6(M)}^2,$$
and dividing through by $\|du\|_{L^6}^2$ gives the desired uniform $L^6$ bound for $|du|$.

With the $L^6$ bound in place, we can simply integrate \eqref{bochner} against $|du|^2$ to see that
$$\int_M |d|du|^2|^2\,dv_g\leq C_n\int (|du|^4+|du|^6)\,dv_g,$$
so the $W^{1,2}$ norm of $|du|^2$ is controlled by the $L^6$ norm, completing the proof.

\end{proof}

With the estimates of Lemma \ref{glob.l6.bd} in place, we can now argue by a contradiction argument (similar in spirit to, but somewhat simpler than, that used in the two-dimensional setting \cite{KS1}) to arrive at the following stabilization result.

\begin{proposition} 
\label{prop:stabilization}
For any Riemannian manifold $(M^n,g)$ of dimension $n\in \{3,4,5\}$, there exists $k_0(M,g)\in \mathbb{N}$ such that for every harmonic map $u\colon M\to \mathbb{S}^k$ with
$$\ind_E(u)\leq k+1,$$
there exists a totally geodesic subsphere $\mathbb{S}^{k_0}\subset \mathbb{S}^k$ of dimension $k_0$ such that $u(M)\subset \mathbb{S}^{k_0}$.
\end{proposition}
\begin{proof} To obtain a contradiction, suppose that the conclusion does not hold: then we can find a sequence of harmonic maps 
$$u_k\colon M\to \mathbb{S}^k$$
with $\ind_E(u_k)\leq k+1$ such that the space spanned by coordinate functions
$$\mathcal{C}_k:=\{\langle e, u_k\rangle \mid e\in \mathbb{R}^{k+1}\}$$
has dimension 
\begin{equation}
\label{dimmk}
m_k:=\dim \mathcal{C}_k\to \infty
\end{equation}
as $k\to\infty$.

Denote $V_k:=|du_k|^2$. Lemma \ref{glob.l6.bd} gives us--for $k$ sufficiently large--an estimate of the form
$\|V_k\|_{W^{1,2}}\leq C$ independent of $k$. In particular, it follows from the Rellich-Kondrachov theorem that $V_k$ is precompact in $L^p$ for any $p<\frac{10}{3}\leq\frac{2n}{n-2}$. Crucially, we observe that since $n\leq 5$, we have $\frac{n}{2}<\frac{2n}{n-2}$, and as a result, up to a choice of a subsequence, we can assume that for some non-negative function $0\leq V\in L^{\frac{n}{2}}(M)$ one has
\begin{equation}
V_k\to V\text{ in }L^{\frac{n}{2}}(M).
\end{equation}
Note, moreover, that 
$$
\int_M V\,dv_g = \lim_{k\to\infty }\int_MV_k\,dv_g= \lim_{k\to\infty }2E(u_k) \geq 2\beta>0,
$$
by virtue of the energy bound $E(u_k)\geq \beta$ from Proposition \ref{skener.prop} in the appendix. Thus, $V\not\equiv 0$.

It follows from the defining equation 
$$
\Delta u_k=|du_k|^2u_k
$$
for sphere-valued harmonic maps that the coordinate functions of $u_k$ are eigenfunctions of the problem
$$
\Delta f =\lambda(V_k)V_k f.
$$
with eigenvalue $\lambda(V_k) = 1$. In particular, the multiplicity of that eigenvalue is at most $m_k$.
It is shown in~\cite[Example 3.19]{GKL} (see also Section~\ref{sec:measures} below) that for any $0\leq W\in L^\frac{n}{2}(M)$, $W\not\equiv 0$ the eigenvalues of the problem
$$
\Delta f =\lambda(W)W f
$$
form a discrete sequence $\{\lambda_m(W)\}_{m=1}^\infty$ with a unique accumulation point at $+\infty$. Furthermore, by~\cite[Proposition 5.1]{GKL} one has that if $W_k\to W$ in $L^\frac{n}{2}(M)$ for $n\geq 3$, then $\lambda_m(W_k)\to\lambda_m(W)$. In particular, $\lambda_m(V_k)\to \lambda_m(V)$. At the same time, by~\eqref{dimmk} one has
$$
\mN_1(V_k):=\#\{i,\,\lambda_i(V_k)\leq 1\}\geq m_k\to\infty
$$
Convergence of eigenvalues implies that
$$
\mN_1(V)\geq\lim_{k\to\infty} \mN_1(V_k) = +\infty.
$$
Therefore, $V$ has infinitely many eigenvalues not exceeding $1$, which contradicts the fact that the only accumulation point of eigenvalues is $+\infty$.

\end{proof}

\section{Connections to spectral shape optimization}
\label{sec:ev}

\subsection{Descriptions of the optimization problem}

Below, unless otherwise specified, $(M,g)$ is a closed Riemannian manifold of dimension $n\geqslant 3$, and all function spaces on $M$ are defined with respect to the metric $g$.

Given a function $V\in L^\infty(M)$,  eigenfunctions of the Schr\"odinger operator $\Delta - V$ satisfy
\begin{equation}
\label{def:mu}
\Delta_g u - Vu = \nu u.
\end{equation}
for some $\nu\in \mathbb{R}$. The corresponding eigenvalues form a sequence
$$
\nu_1(V)<\nu_2(V)\leqslant \nu_3(V)\leqslant \ldots \nearrow \infty,
$$
where each number is repeated according to its multiplicity. The eigenvalues $\nu_m(V)$ can be equivalently defined variationally via
\begin{equation}
\label{var:mu}
\nu_m(V) = \inf_{F_m} \sup_{u\in F_m\setminus \{0\}} \frac{\int_M|du|_g^2 - Vu^2\,dv_g}{\int_M u^2\,dv_g},
\end{equation}
where $F_m\subset W^{1,2}(M)$ is a $m$-dimensional subspace. 

In the present section we introduce the quantity
\begin{equation}
\label{def:optV2}
\optV_m(M,g) :=\sup_{V,\,\nu_{m+1}(V)\geqslant 0} \int V\,dv_g,
\end{equation}
and consider the associated maximization problem for potentials $V$ with $\nu_{m+1}(V)\geq 0$. To simplify notation, we let $\spaceV_m\subset L^{\infty}(M,g)$  be the set of all $V$ such that $\nu_{m+1}(V)\geq 0$. This way 
$$
\optV_m(M,g) :=\sup_{V\in \spaceV_{m}} \int V\,dv_g.
$$

A version of this problem was studied in~\cite{GNS} in relation to the lower bound on the number of negative eigenvalues of the operator $\Delta_g - V$. Let $\mN(V) = \#\{i,\,\nu_i(V)<0\}$. Then for any $V$ with $\mN(V) = m$ one has $V\in\spaceV_m$, so that 
$$
\frac{\optV_m}{m^{2/n}}\geq  \frac{1}{\mN(V)^{2/n}}\int V\,dv_g.
$$
Thus, $C_{opt} = \inf_m \frac{m^{2/n}}{\optV_m}$ is the optimal constant in the following inequality
\begin{equation}
\label{ineq:mNV}
\mN(V)^{2/n}\geqslant C_{opt}\int_M V\,dv_g.
\end{equation}
An inequality of the form~\eqref{ineq:mNV} for non-negative potentials $V\geq 0$ was proved in~\cite{GNY}, and was later shown to hold with the same constant for all $L^\infty$ potentials in~\cite{GNS}. In particular, it follows from~\cite{GNS,GNY} that $C_{opt}>0$.  In fact, we will see below that potentials realizing $\optV_m$ (when they exist) arise as energy densities of harmonic maps to spheres and are therefore non-negative.

There is an equivalent definition of $\optV_m$ in terms of Laplacian eigenvalues with densities. Let $0\leq\beta\in L^\infty(M)$ be a non-negative density $\beta\not\equiv 0$. Consider the weighted problem
\begin{equation}
\label{def:lambda}
\Delta_g f =\lambda\beta f. 
\end{equation}
The eigenvalues form a sequence
$$
0 = \lambda_0(\beta)<\lambda_1(\beta)\leqslant \lambda_2(\beta)\leqslant \ldots \nearrow \infty,
$$
where eigenvalues are written with multiplicity.

The eigenvalues $\lambda_k(\beta)$ admit the variational characterization
\begin{equation}
\label{var:lambda}
\lambda_{m-1}(\beta) = \inf_{F_m} \sup_{u\in F_m\setminus \{0\}} \frac{\int_M|du|_g^2\,dv_g}{\int_M u^2\beta\,dv_g},
\end{equation}
where $F_m\subset W^{1,2}(M)$ is a $m$-dimensional subspace. Comparing variational characterizations~\eqref{var:mu} and~\eqref{var:lambda}, it is easy to see that $V_{m,\beta}=\lambda_k(\beta)\beta\in\spaceV_m$. Furthermore,
$$
\int V_{m,\beta}\,dv_g = \lambda_{m}(\beta)\int_M\beta\,dv_g,
$$
which suggests the relation between $\optV_m$ and maximizing the r.h.s over $\beta\in L^\infty$. The main difference between the two problems is the fact $V_{m,\beta}$ is always non-negative, so $\{V_{m,\beta},\,\beta\in L^\infty\}$ is a proper subset of $\spaceV_m$. Despite this, the following holds.

\begin{proposition}
\label{prop:equivalent}
One has
$$
\optV_m = \sup_{0\leq\beta\in L^\infty(M)}\lambda_m(\beta)\int_M\beta\,dv_g
$$
\end{proposition}

\begin{proof}
In~\cite{GNS} it is shown that the maximizers of the auxiliary optimization problem
$$
\optV_m^N = \sup_{V\in\spaceV_m,\, \|V\|_\infty\leq N}\int_M V\,dv_g
$$ 
exist and are non-negative for $N$ sufficiently large. Since $\optV_m = \sup_N\optV_{m}^N$ this is enough to finish the proof.
A detailed argument is carried out in~\cite[Section 2]{KNPP2}. Note that that the paper~\cite{KNPP2} is concerned with eigenvalues of surfaces, but the arguments in Section 2 do not rely on the dimension of the manifold and thus can be repeated verbatim in our case.
\end{proof}

\subsection{Admissible measures} 
\label{sec:measures}
In many situations it is convenient to relax the regularity conditions on the potential $V$ or density $\beta$. In dimension two, for example, this is an important ingredient in establishing the explicit link between optimization of Steklov and Laplace eigenvalues on surfaces, see~\cite{GL,GKL, KS1,KS2}.

\begin{definition}[See~\cite{GKL}]
A (positive) Radon measure is called {\em admissible} if the identity map on $C^\infty(M)$ can be extended to a compact operator $W^{1,2}(M,g)\to L^2(\mu)$.
\end{definition}

One can then define variational eigenvalues $\lambda_k(\mu)$ by a formula similar to~\eqref{var:lambda}
\begin{equation*}
\lambda_{m-1}(\mu) = \inf_{F_m} \sup_{u\in F_m\setminus \{0\}} \frac{\int_M|du|_g^2\,dv_g}{\int_M u^2\,d\mu},
\end{equation*}
where $F_m\subset W^{1,2}(M)$ is a $m$-dimensional subspace. Eigenvalues form a discrete sequence accumulating to infinity and there exist corresponding eigenfunctions satisfying $\Delta_g f\,dv_g = \lambda f\,d\mu$ in the weak sense, see~\cite{GKL}.

\begin{example}
It is shown in~\cite[Example 3.19]{GKL} that $\mu = \beta\,dv_g$ is admissible provided $0\leq\beta\in L^{\frac{n}{2}}(M)$. The eigenvalues $\lambda_m(\mu)$ can be equivalently defined via equation~\eqref{def:lambda}. In particular, the eigenvalues $\lambda_m(\beta)$ make sense for $\beta\in L^{\frac{n}{2}}(M)$.  
\end{example}

\begin{example}
\label{ex:Steklovdensity}
Assume that $\bd M\ne\varnothing$. It is shown in~\cite[Lemma 3.3]{KL} that $\mu = \rho\,ds_g$ is admissible, where $0\leq\rho\in L^{n-1}(\bd M)$ and $ds_g$ is the volume measure of $\bd M$. The quantities $\lambda_m(\mu)$ are the eigenvalues of the Steklov problem with density
\begin{equation*}
\begin{cases}
\Delta_g u = 0  &\text{ in $M$};\\
\bd_n u = \lambda\rho u &\text{ on $\bd M$}.\\
\end{cases}
\end{equation*}
Following the standard convention used for Steklov eigenvalues, we write $\sigma_m(\rho):=\lambda_m(\rho\,ds_g)$.
\end{example}

One can also define the appropriate relaxation of the eigenvalue equation~\eqref{def:mu}. Namely, a signed Radon measure $\xi$ is called admissible if its total variation $|\xi|$ is admissible. One can then define eigenvalues of the Schr\"odinger operator $\Delta_g - \xi$ variationally as
\begin{equation*}
\nu_m(\xi) = \inf_{F_m} \sup_{u\in F_m\setminus \{0\}} \frac{\int_M|du|_g^2\,dv_g - \int_M u^2\,d\xi}{\int_M u^2\,dv_g},
\end{equation*}
where $F_m\subset W^{1,2}(M)$ is a $m$-dimensional subspace. 
%
%

\begin{example}
\label{ex:boundarypotential}
According to Example~\ref{ex:Steklovdensity} the measure $\xi= v\,ds_g$ is admissible as soon as $v\in L^{n-1}(\bd M)$. Quantities $\nu_m(\xi)$ are eigenvalues of the Robin problem
\begin{equation*}
\begin{cases}
\Delta_g u = \nu u  &\text{ in $M$};\\
\bd_n u = v u &\text{ on $\bd M$}.\\
\end{cases}
\end{equation*}
We say $v\in\optV_{m}^\bd$ if $\nu_{m+1}(\xi)\geq 0$.
\end{example}

In what follows, we denote by $\mathcal{E}_k(M,g)$ the energies $E(u_k)$ of the harmonic maps $u_k\colon M\to \mathbb{S}^k$ in Theorem \ref{sk.minmax.thm}. By construction, these energies have the min-max characterization
$$\mathcal{E}_k:=\lim_{\epsilon\to 0}\mathcal{E}_{k,\epsilon},$$
where
$$\mathcal{E}_{k,\epsilon}:=\inf_{(u_y)\in \Gamma_k(M)}\max_{y\in \mathbb{B}^{k+1}}E_{\epsilon}(u_y),$$
$$\Gamma_k(M):=\{\mathbb{B}^{k+1}\ni y\mapsto u_y\in W^{1,2}(M,\mathbb{R}^{k+1})\mid u_y\equiv y\text{ for }y\in \mathbb{S}^k\},$$
and $E_{\epsilon}\colon W^{1,2}(M,\mathbb{R}^{k+1})\to \mathbb{R}$ is a Ginzburg-Landau-type functional as described in Section \ref{sec:minmax}. With this notation in place, we can now state the main theorem of this section.

\begin{theorem}
\label{ev:main_theorem}
Let $(M,g)$ be a closed Riemannian manifold of dimension $3\leq n \leq 5$. Then
\begin{itemize}
\item[(i)] There exists $k_0 = k_0(M,g)$ such that $\mE_k = \mE$ for all $k\geq k_0$. Furthermore, $\optV_1(M,g) =2 \mE(M,g)$.
\item[(ii)] There exists a smooth potential $V$ with $\nu_2(V) = 0$ such that $\optV_1 = \int_M V\,dv_g$. In particular, there exists a smooth density $\beta\geq 0$ such that $\optV_1 = \lambda_1(\beta)\int_M\beta\,dv_g$.
\item[(iii)] For any admissible signed measure $\xi$ with $\nu_2(\xi)\geq 0$ one has $\nu(M)\leq \optV_1$ with equality iff $\nu_2(\xi) = 0$ and $\xi = |du|^2_g\,dv_g$, where $u\colon (M,g)\to\mathbb{S}^k$ is a smooth harmonic map.
In particular, for any admissible non-negative measure $\mu$ one has $\lambda_1(\mu)\mu(M)\leq\optV_1$ with equality iff 
$\lambda_1(\mu)\mu = |du|^2_g\,dv_g$, where $u\colon (M,g)\to\mathbb{S}^k$ is a smooth harmonic map.
\end{itemize}
\end{theorem}

We conclude this section with the following proposition, which is one of the main ingredients in proving Theorem~\ref{ev:main_theorem}.(iii)

\begin{proposition}
\label{prop:harm_reg}
Let $u\colon (M,g)\to\mathbb{S}^k$ be a weakly harmonic map such that $\mu = |du|_g^2\,dv_g$ is an admissible measure. Then $u$ is smooth.
\end{proposition}

\begin{proof} The admissibility of the energy density measure $\mu=|du|^2dv_g$ is evidently equivalent to the statement that, for any bounded sequence $\varphi_j\in W^{1,2}(M,g)$ satisfying $\varphi_j\rightharpoonup 0$ as $j\to\infty$, we have 
$$\lim_{j\to\infty}\int_M \varphi_j^2|du|^2dv_g=0.$$

Now, for an arbitrary sequence of points $p_j\in M$ and a vanishing sequence of radii $r_j\to 0$, consider a sequence $\varphi_j\in C_c^{\infty}(B_{2r_j}(p_j))$ such that $\varphi_j\equiv 1$ on $B_{r_j}(p_j)$ and $|d\varphi_j|\leq \frac{C}{r_j}$; then it is easy to see that
$$\int_M \varphi_j^2dv_g\leq C(M)r_j^n\text{ and }\int_M |d\varphi_j|_g^2dv_g\leq C(M)r_j^{n-2}.$$
In particular, setting $\psi_j:=r_j^{\frac{2-n}{2}}\varphi_j$, we obtain a sequence $\psi_j\in W^{1,2}(M)$ with 
$$\|\psi_j\|_{L^2}^2\leq C(M)r_j^2\text{ and }\|d\psi_j\|_{L^2}^2\leq C(M),$$
so that $\psi_j$ is bounded in $W^{1,2}(M)$ and $\psi_j\rightharpoonup 0$ as $j\to \infty$. Hence, by the admissibility of $|du|_g^2dv_g,$ we have
$$\lim_{j\to\infty}\int \psi_j^2 |du|_g^2dv_g=0.$$
In particular, since $\psi_j^2=r_j^{2-n}\varphi_j^2\equiv r_j^{2-n}\text{ on }B_{r_j}(p_j),$ it follows that
$$\lim_{j\to\infty}r_j^{2-n}\int_{B_{r_j}(p_j)}|du|_g^2dv_g=0$$
for every sequence $p_j\in M$ and $r_j\to 0$. In other words, we have
\begin{equation}\label{van.dens}
\lim_{r\to 0} \sup_{p\in M} r^{2-n}\int_{B_r(p)}|du|_g^2dv_g=0.
\end{equation}
Finally, it follows from work of Rivi\`ere-Struwe \cite[Theorem 1.1]{RS08} that any weakly harmonic map $u\colon M\to N$ satisfying \eqref{van.dens} must indeed be smooth, giving the desired result.

\end{proof}

\subsection{Spectral index} Let $u\colon (M,g)\to\mathbb{S}^k$ be a harmonic map. Consider the associated Schr\"odinger operator $\mL_u = \Delta_g  - |du|_g^2$.

\begin{definition}
The {\em spectral index} $\ind_S(u)$ of $u$ is the number of negative eigenvalues of $\mL_u$. Similarly, the {\em spectral nullity} 
$\nul_S(u)$ of $u$ is the dimension of the kernel of $\mL_u$.
\end{definition}

The proof of the following proposition is analogous to the case $\dim M=2$ covered in~\cite[Propositions 1.7 and 3.11]{KRP2}.

\begin{proposition} 
\label{prop:indS}
For any smooth harmonic map $u\colon (M,g)\to\mathbb{S}^k$ one has 
$$
\ind_S(u)\geqslant \frac{\ind_E(u)}{k+1}.
$$
Furthermore, let $i_{k,k'}\colon \mathbb{S}^k\hookrightarrow\mathbb{S}^{k'}$ be a totally geodesic embedding. Then $i_{k,k'}\circ u$ is harmonic, $\ind_S(i_{k,k'}\circ u) = \ind_S(u)$ and
\begin{equation}
\label{eq:indS}
\ind_E(i_{k,k'}\circ u) = \ind_E(u) + (k'-k)\ind_S(u).
\end{equation}
\end{proposition}

\subsection{Proof of Theorem~\ref{ev:main_theorem}}

The general strategy is the same the one employed in the proofs of~\cite[Theorems 1.3 and 1.4]{KS1}.

\begin{proposition}
\label{prop:upper_bound}
For any $k\geqslant 3$ and any signed admissible measure $\xi$ satisfying $\nu_2(\xi)\geqslant 0$ there exists a map $u\in W^{1,2}(M,\mathbb{S}^k)$ such that $E_g(u)\leq \mE_k$ and
$$
\int u \varphi_0\,dv_g = 0,
$$
where $\varphi_0$ is a $\nu_1(\xi)$-eigenfunction.

In particular, 
\begin{equation}
\label{ineq:main_ub}
\xi(M)\leqslant 2E_g(u) 
\end{equation}
with equality iff $\nu_2(\xi) = 0$, $u$ is a smooth harmonic map of spectral index $1$ and $\xi = |du|_g^2\,dv_g$.
\end{proposition}

\begin{proof}

 If $\int_M\varphi_0\,dv_g = 0$, then one can take $u$ to be a constant map. Otherwise we can choose $\varphi_0$ so that $\int_M\varphi_0\,dv_g = 1$.
For each $\varepsilon>0$ let $(\mathbb{B}^{k+1}\ni y\mapsto u^\varepsilon_y)\in\Gamma_k$ be a family such that $\sup_yE_\varepsilon(u^\varepsilon_y)\leqslant \mE_{k,\varepsilon}+\varepsilon$.
Define the map
$$
I(y):=\int_M u^\eps_y\varphi_0\,dv_g.
$$
Since $u^\eps_y\equiv y$ for $y\in\bd\mathbb{B}^{k+1} = \mathbb{S}^k$, the restriction of $I$ to $\bd\mathbb{B}^{k+1}$ is the identity map. A standard application of Brouwer fixed point theorem yields the existence of $y_\eps\in \mathbb{B}^{k+1}$ such that $I(y_\eps) = 0$. Setting $u_\eps:=u^\eps_{y_\eps}\in W^{1,2}(M,\mathbb{R}^{k+1})$, we then have

$$
\int_M u_\varepsilon\varphi_0\,dv_g=0, \quad E(u_\varepsilon)\leqslant E_\varepsilon(u_\varepsilon)\leqslant \mE_{k,\varepsilon}+\varepsilon\to \mE_k,
$$
$$
 \int_M d_{\mathbb{S}^k}(u_\eps)^2\,dv_g \leqslant \varepsilon^2c_1^{-1}E_\varepsilon(u_\varepsilon)\to 0,
$$
where $c_1$ is the constant in~\eqref{w.dist.comp}. Furthermore,
\begin{equation*}
\begin{split}
\int_M |u_\eps|^2\,dv_g \leq \int_M \max(W(u_\eps),R_0^2)\,dv_g \leq \eps^2E_\eps(u_\eps) + R_0^2\vol(M,g),
\end{split}
\end{equation*}
which implies that $\|u_\eps\|_{L^2(M,g)}$ is bounded.
Thus, there exists a subsequence $u_{\varepsilon_i}$ which converges weakly in $W^{1,2}(M,\mathbb{R}^{k+1})$ and strongly in $L^2(M,\mathbb{R}^{k+1})$ to $u\in W^{1,2}(M,\mathbb{R}^{k+1})$. Since $\varphi_0\in W^{1,2}(M)$ one has 
\begin{equation}
\label{eq0}
\int_M u\varphi_0\,dv_g = 0.
\end{equation}
Since $d_{\mathbb{S}^k}(u)^2 = |1-|u||^2\leqslant 2(|1-|u_\eps||^2+|u_\eps-u|^2)$, one also has
$$
\int_M d_{\mathbb{S}^k}(u)^2\,dv_g = 0,
$$
and, therefore, $u\in W^{1,2}(M,\mathbb{S}^k)$. Since the energy is lower-semicontinuous with respect to weak convergence in $W^{1,2}$ one has
$$
E(u)\leqslant\lim_{\varepsilon\to 0}E(u_\varepsilon)\leqslant \mE_k. 
$$

In particular, by~\eqref{eq0} the components $u^i$ of $u$ can be used as test-functions for $\nu_2(\xi)$,
$$
\int_M |du^i|^2_g\,dv_g - \int_M |u^i|^2\,d\xi\geq \nu_2(\xi)\geqslant 0.
$$
Summing over $i=1,\ldots, k+1$ yields
\begin{equation}
\label{ineq:aux1}
\xi(M)\leq \int |du|^2_g\,dv_g = 2E(u).
\end{equation}
The equality occurs iff $\nu_2(\xi) = 0$ and $u^i$ are $\nu_2(\xi)$-eigenfunctions. It then follows that for any $v\in C^\infty(M)$ one has
\begin{equation*}
\begin{split}
\int_M v\,d\xi = &\int_M v|u|^2\,dxi = \sum_{i=1}^{k+1}\int_M \la du^i, d(u^iv)\ra\,dv_g =\\
&\int_M|du|^2v + \la dv,\frac{1}{2}d|u|^2\ra\,dv_g =  \int v|du|^2_g\,dv_g,
\end{split}
\end{equation*}
i.e. $\xi =|du|^2_g\,dv_g$, which implies that $\Delta_g u = |du|_g^2u$ in the weak sense, i.e. $u$ is a weak harmonic map. Since $\xi$ is admissible by assumption, Proposition~\ref{prop:harm_reg} implies that $u$ is smooth.
\end{proof}

Note that Proposition~\ref{prop:upper_bound} implies the bound $\optV_1\leq 2\mE_k$ for all $k\geq 3$ and any $n\geq 3$. To prove the opposite inequality, assume $3\leq n\leq 5$ and
let $u_k\colon (M,g)\to \mathbb{S}^k$ be a smooth harmonic map given by Theorem~\ref{sk.minmax.thm}, so that
$$
E(u_k) = \mE_k;\qquad \ind_S(u_k)\leqslant k+1.
$$
Then by Proposition~\ref{prop:upper_bound} (or by~\eqref{unif.en.bds}), the maps $u_k$ satisfy the conditions of Proposition~\ref{prop:stabilization}. Let $k_0$ be as in Proposition~\ref{prop:stabilization} and let $k>2k_0+1$, then $u_{k}= i_{k_0,k}\circ v$ for some harmonic $v$. By~\eqref{eq:indS} one has
$$
k+1\geq \ind_E(u_{k}) \geq  (k-k_0)\ind_S(u_{k})> \frac{1}{2}(k+1)\ind_S(u_k).
$$
Since $\ind_S$ is integer, this implies $\ind_S(u_k) = 1$. Thus, $\lambda_2(|du_k|_g^2) = 0$ and one has
$$
2\mE_k = 2E_g(u_k) = \int|du_k|^2_g\,dv_g\leq\optV_1\leq2\mE_k,
$$
which implies that all inequalities are equalities. In particular, $\optV_1 = 2\mE_k$ for $k>2k_0+1$ and $|du_k|_g^2\in\spaceV_1$ is a smooth potential on which $\optV_1$ is attained. This completes the proofs of Theorem~\ref{ev:main_theorem}.(i) and (ii).

Let $\xi$ be any admissible measure with $\nu_2(\xi)\geq 0$ and such that $\xi(M) = \optV_1$. Let $k$ be large enough so that $\mE_k = \optV_1$ and let $u$ be the map given by Proposition~\ref{prop:upper_bound}. By inequality~\eqref{ineq:main_ub} one has
$$
\optV_1 = \xi(M) \leq 2E_g(u)\leq \mE_k = \optV_1.
$$
Thus, all inequalities are equalities and Proposition~\ref{prop:upper_bound} implies that $\nu_2(\xi) = 0$, $u$ is a smooth harmonic map with $\ind_S(u) = 1$ and $\xi = |du|_g^2\,dv_g$. This completes the proof of Theorem~\ref{ev:main_theorem}.(iii).

\subsection{Relation with conformal volume.} Particular examples of families 
$\mathbb{B}^{k+1}\ni y\mapsto u_y$ such that  $u_y\equiv y\text{ for }y\in \mathbb{S}^k$ can be obtained by post-composition of a given map with conformal automorphisms of $\mathbb{S}^k$. Namely, for any $y\in\mathrm{int}\left(\mathbb{B}^{k+1}\right)$ the map $G_y(x)=\frac{1-|y|^2}{|x+a|^2}+a$ is a conformal diffeomorphism $\mathbb{S}^k\to\mathbb{S}^k$. If $u=u_0\colon (M,g)\to \mathbb{S}^k$ is a conformal map, then it is easy to check that $u_y:=G_y\circ u_0$ extends to a weakly continuous family satisfying $u_y\equiv y\text{ for }y\in \mathbb{S}^k$. This family is often referred to as the {\em canonical family of $u_0$}. It was used by Li and Yau~\cite{LiYau} (see also~\cite{ESIconfvolume}) to define the notion of {\em conformal volume} as follows. One first defines conformal volume of the map $u$ to be
$$
V_c(k, u) = \sup_y\vol(M,(u_y)^*g_{\mathbb{S}^k}).
$$
Taking the infimum over all conformal maps $u$ gives the ($k$-)conformal volume of $(M,[g])$
$$
V_c(k,M,[g]) = \inf_{u}V_c(k,u).
$$
The connection between conformal volume and minimal submanifolds is formulated in the following theorem. It was proved by Li and Yau ~\cite{LiYau} for surfaces, by El Soufi and Ilias~\cite{ESIconfvolume} for higher dimensional manifolds.
\begin{theorem}[Li, Yau~\cite{LiYau}; El Soifi-Ilias~\cite{ESIconfvolume}]
\label{thm:conf_volume}
Let $u\colon M\to \mathbb{S}^k$ be a minimal immersion. Assume WLOG that $u$ is linearly full, i.e. the image $u(M)$ spans $\mathbb{R}^{k+1}$. Then for any $y\ne 0$ one has 
$$
\vol(M,(u_y)^*g_{\mathbb{S}^k})< \vol(M,(u)^*g_{\mathbb{S}^k})
$$
unless $M=\mathbb{S}^k$ and $u = G_{y'}$ up to a rotation of $\mathbb{R}^{k+1}$. In particular, 
$$
V_c(k,M,[u^*g_{\mathbb{S}^k}]) \leq \vol(M,(u)^*g_{\mathbb{S}^k}).
$$
\end{theorem}

Observe that for any canonical family $\{u_y\}$ one has
\begin{equation}
\label{ineq:n-energy}
\begin{split}
2E(u_y) &= \int_M|du_y|^2_g\,dv_g\leq n\left(n^{\frac{n}{2}}\int_M|du_y|^n_g\,dv_g\right)^\frac{2}{n}\vol(M,g)^{\frac{n-2}{n}} =\\
&= n \left(\vol(M,(u_y)^*g_{\mathbb{S}^k})\right)^\frac{2}{n}\vol(M,g)^{\frac{n-2}{n}},
\end{split}
\end{equation}
where we first used H\"older inequality and then conformality of $u_y$. Taking the supremum over $y$ and then the infimum over all canonical families suggests the following inequality
\begin{equation}
\label{bd:mEVc}
2\mE_k(M,g)\leq n \left(V_c(k,M,[g])\right)^\frac{2}{n}\vol(M,g)^{\frac{n-2}{n}}.
\end{equation}
However, note that the canonical family is only weakly continuous, so one can not use it directly in the definition of $\mE_k(M,g)$. Nevertheless, a mollification procedure similar to that employed in the proof of Lemma~\ref{lemma:energy_upbound} (see also~\cite[Proposition 3.3]{KS1}) yields inequality~\eqref{bd:mEVc}.
Combining Proposition~\ref{prop:upper_bound} with~\eqref{bd:mEVc} we obtain
\begin{equation}
\label{ineq:optVVc}
\optV_1(M,g)\leq 2\mE_k(M,g)\leq n \left(V_c(k,M,[g])\right)^\frac{2}{n}\vol(M,g)^{\frac{n-2}{n}},
\end{equation}
which together with Theorem~\ref{thm:conf_volume} gives the inequality in Theorem~\ref{thm:ev_exmp}.

To study the equality case, assume that there is a minimal immersion $u\colon M\to\mathbb{S}^k$ such that $u^*g_{\mathbb{S}^k}\in [g]$, in particular, 
$$
V_c(k,M,[g])\leq \vol(M,u^*g_{\mathbb{S}^k}).
$$
Assume that there exists an admissible signed measure $\xi$ with $\nu_2(M,g,\xi)\geq 0$ such that there is an equality
$$
\xi(M) = n \left(\vol(M,u^*g)\right)^\frac{2}{n}\vol(M,g)^{\frac{n-2}{n}} = 2\mE_k(M,g).
$$
Hersch's trick implies that there exists $u_y$ in the canonical family of $u$ such that components of $u$ are orthogonal to a $\nu_1(M,g,\xi)$-eigenfunction. Similarly to~\eqref{ineq:aux1} one obtains 
$$
\xi(M)\leq 2E_g(u_y)
$$
with equality only if $\nu_2(\xi)=0$.
Combining with~\eqref{ineq:n-energy} one obtains
\begin{equation*}
\begin{split}
2\mE_k(M,g)=\xi(M) & \overset{\raisebox{.5pt}{\textcircled{\raisebox{-.9pt} {1}}}}{\leq} 2E_g(u_y)  \overset{\raisebox{.5pt}{\textcircled{\raisebox{-.9pt} {2}}}}{\leq}  n \left(\vol(M,(u_y)^*g_{\mathbb{S}^k})\right)^\frac{2}{n}\vol(M,g)^{\frac{n-2}{n}} \\
&\overset{\raisebox{.5pt}{\textcircled{\raisebox{-.9pt} {3}}}}{\leq} n \left(\vol(M,u^*g_{\mathbb{S}^k})\right)^\frac{2}{n}\vol(M,g)^{\frac{n-2}{n}} = 2\mE_k(M,g).
\end{split}
\end{equation*}
In particular, all inequalities are equalities. Equality in \raisebox{.5pt}{\textcircled{\raisebox{-.9pt} {1}}} implies that $u_y$ is a harmonic map of spectral index $1$ and $\xi = |du_y|_g^2\,dv_g$. Equality in \raisebox{.5pt}{\textcircled{\raisebox{-.9pt} {2}}} (or, equivalently,~\eqref{ineq:n-energy}) implies that $|du_y|_g^2$ is constant and, in particular, $\xi =b\,dv_g$ for some $b\in\mathbb{R}$. Thus, $u_y$ is a conformal harmonic map with constant energy density, therefore, a minimal immersion such that $g = a(u_y)^*g_{\mathbb{S}^k}$. Since $\ind_S(u_y) = 1$, components of $u_y$ are the first eigenfunctions.
 If $(M,[g])\not=(\mathbb{S}^k,[g_{\mathbb{S}^k}])$ and $u\not=G_{y'}$ for some conformal automorphism $G_{y'}$, then equality in \raisebox{.5pt}{\textcircled{\raisebox{-.9pt} {3}}} implies $y=0$. Otherwise, $u_y = G_{y''}$ for some $y''$ and the equality statement in Theorem~\ref{thm:ev_exmp} follows. 

\subsection{Critical potentials}

In this section we collect results on critical points of functionals associated with the optimization problems studied above. The techniques of the proofs are standard and go back to the papers~\cite{NadirashviliT2, ESIextremal, FS:extremal}. We also make use of computations in~\cite{KM} which allow for a concise proof.

\begin{definition}
\label{def:1}
Let $0<\beta\in C^\infty(M)$. We say that $\beta$ is a critical density for the functional 
\begin{equation}
\label{def:functional}
F_m(\beta) = \lambda_m(\beta)\int_M\beta\,dv_g
\end{equation}
if for any smooth deformation $\beta(t)>0,$ $\beta(0) = \beta$, $t\in(-\eps,\eps)$ one has
\begin{equation}
\label{def:critical}
F_m(\beta(t))\leq F_m(\beta) + o(t)\qquad\text{or}\qquad F_m(\beta(t))\geq F_m(\beta) + o(t)
\end{equation}
as $t\to 0$.
\end{definition} 

\begin{proposition}
\label{prop:critical1}
Let $0<\beta\in C^\infty(M)$ be a critical density for the functional~\eqref{def:functional}.
Then there exists a smooth harmonic map $u\colon (M,g)\to \mathbb{S}^k$ such that $\lambda_m(\beta)\beta = |du|_g^2$ and the components of $u$ are $\lambda_m(\beta)$-eigenfunctions. In particular, $\ind_S(u)\leq m$. 

Conversely, let $u\colon (M,g)\to\mathbb{S}^k$ be a smooth harmonic map. Then $|du|_g^2$ is a critical density for the functional~\eqref{def:functional} with $m = \ind_S(u)$.
\end{proposition}
\begin{proof}
 The general strategy of the proof for such results is the same. For a smooth deformation $\beta(t)$ one first computes the derivative of~\eqref{def:functional} at a generic point $t$. For critical $\beta$ the derivative at $t=0$ usually does not exist, but one can use criticality to show that a certain quadratic form $A(\beta'(0))$ on $\lambda_m(\beta)$-eigenspace is sign indefinite for all $\beta'(0)$. After that, an application of Hahn-Banach separation theorem yields a quadratic relation between elements of the eigenspace.

For the functional~\eqref{def:functional}, this was essentially carried out in~\cite{KM}, where the same functional is considered on a larger space: namely, the metric $g$ is also allowed to vary in the conformal class. In particular, 
 the computation in the proof of~\cite[Theorem 8]{KM} gives that if $\int_M\beta'(0)\,dv_g = 0$, then the quadratic form
$$
A(\beta'(0))[\phi,\psi] = -\lambda_m(\beta)\int_M\beta'(0)\phi\psi\,dv_g
$$
is sign indefinite.
Combining with~\cite[Lemma 1]{KM}, 
we deduce existence of $\lambda_m(\beta)$-eigenfunctions $u_1,\ldots, u_{k+1}$ such that
$\sum_{i=1}^{k+1} (u_i)^2 = 1$. Setting $u = (u_1,\ldots, u_{k+1})$ one has
$$
0 = \frac{1}{2}\Delta_g(|u|^2) = \lambda_m(\beta)\beta|u|^2 - |du|_g^2,
$$
which shows that $\lambda_m(\beta)\beta = |du|_g^2$. Since $\beta$ is smooth, $u$ is smooth and, furthermore, 
$$
\Delta_g u = \lambda_m(\beta)\beta u = |du|_g^2u,
$$
i.e. $u$ is a smooth harmonic map.

The proof of the converse is by now standard, we refer to~\cite{KM} for details. 
\end{proof}

We state the analogous theorem for potentials. This formulation has an advantage that there is no restriction on the sign of potential and, for example, can be used to study the functional~\eqref{def:functional} for $\beta\geq 0$. Note that the definition of a critical potential has to be modified due to the fact that the space $\spaceV_m$ is much more complicated than $C^\infty(M)$, so one has to allow one-sided deformations.
\begin{definition}
Let $V\in C^\infty(M)\cap \mathcal{P}_m$. We say that $V$ is a critical potential for the functional 
\begin{equation}
\label{def:functional2}
V\mapsto \int_MV\,dv_g
\end{equation}
in $\spaceV_m\cap C^\infty(M)$ if for any smooth deformation $[0,\epsilon)\ni t\mapsto V(t)\in \spaceV_m\cap C^\infty(M)$ with $V(0)=V$, one has 
\begin{equation}
\label{def:critical2}
\int V(t)\,dv_g\leq\int V\,dv_g + o(t)
\end{equation}
as $t\to 0+$.
\end{definition}
\begin{remark}
By considering the deformation $V(t) = V-t$, it is easy to see that the opposite of the inequality in~\eqref{def:critical2} can never be enforced for arbitrary deformations in $\mathcal{P}_m$.
\end{remark}

\begin{proposition}
\label{prop:critical2}
Let $V\in\spaceV_m\cap C^\infty(M)$ be a smooth potential critical for the functional~\eqref{def:functional2}
 in $\spaceV_m\cap C^\infty(M)$. Then there exists a smooth harmonic map $u\colon (M,g)\to \mathbb{S}^k$ such that $|du|_g^2 = V$ and  $\nu_m(V) = 0$. In particular, $\ind_S(u)\leq m$. 

Conversely, let $u\colon (M,g)\to\mathbb{S}^k$ be a smooth harmonic map. Then $|du|_g^2\in\spaceV_{\ind_S(u)}\cap C^\infty(M)$ is a critical potential for the functional~\eqref{def:functional2} in $\spaceV_{\ind_S(u)}\cap C^\infty(M)$.
\end{proposition}
\begin{proof}
Let $V\in\spaceV_m\cap C^\infty(M)$ be critical. If $\nu_m(V)>0$, then the deformation $V(t) = V+t\in \spaceV_m\cap C^\infty(M)$ for small $t>0$. For this deformation, condition~\eqref{def:critical2} is not satisfied and, thus, $\nu_m(V) = 0$.
 
Let $W\in C^\infty(M)$ be such that $\int W\,dv_g = 0$. Consider the
 deformation $V(t) = V + tW$ for $0\leq t<0$. Let 
$$
\nu_{i-1}(V) < 0 = \nu_i(V) = \ldots = \nu_m(V) = \ldots =\nu_{i+p}(V)<\nu_{i+p+1}(V),
$$ 
so that $\nu_m(V)$ has multiplicity $p+1$. An application of Rellich-Kato perturbation theory~\cite{Kato} (or~\cite[Lemma 3.2]{GNS}) yields the following expressions for $j=0,\ldots p$
\begin{equation*}
\nu_{i+j}(V(t)) = 
\gamma_jt +o(t),
\end{equation*}
where $\gamma_j$ are eigenvalues of the quadratic form $A(W)$ on $\nu_m(V)$-eigenspace
\begin{equation}
\label{Aform}
A(W)[\phi,\psi] = -\int_MW\phi\psi\,dv_g
\end{equation}
arranged in the increasing order. If $A(W)>0$, then for some small $\delta>0$ one has $A(W+\delta)>0$ and, as a result, $V+t(W+\delta)\in \spaceV_m$ for small $t>0$. But
$$
\int_M V+t(W+\delta)\,dv_g = \int_MV\,dv_g + t\vol(M,g)>\int_M V + o(t),
$$
 which contradicts~\eqref{def:critical2}. If $A(W)<0$, then $A(-W)>0$ and the same argument yields a contradiction. Thus, $A(W)$ is sign indefinite and the same application of Hahn-Banach separation theorem as in the proof of Proposition~\ref{prop:critical1} yields  
 the existence of $\nu_m(V)$-eigenfunctions $u_1,\ldots, u_{k+1}$ such that
$\sum_{i=1}^{k+1} (u_i)^2 = 1$. Setting $u = (u_1,\ldots, u_{k+1})$ one has
$$
0 = \frac{1}{2}\Delta_g(|u|^2) = V|u|^2 - |du|_g^2,
$$
which shows that $V = |du|_g^2$. Since $V$ is smooth, $u$ is smooth and, furthermore, 
$$
\Delta_g u = Vu = |du|_g^2u,
$$
i.e. $u$ is a smooth harmonic map.

For the converse, let $u$ be a harmonic map and let $V=|du|_g^2$. This implies that the form $A(W)$ is sign indefinite for all $W$ such that $\int_M W\,dv_g = 0$. Assume $V$ is not critical in $\spaceV_{\ind_S(u)}$. Then there exists a deformation $V(t)$ such that 
$$
a=\frac{1}{\vol(M,g)}\int_MV'(0)\,dv_g>0
$$
and $V(t)\in\spaceV_{\ind_S(u)}$. The latter implies that $A(V'(0))\geq 0$, therefore, $A(V'(0)-a) = A(V'(0))+aI>0$. At the same time, $\int V'(0)-a\,dv_g = 0$ and, thus, $A(V'(0)-a)$ has to be sign indefinite, a contradiction.
 
\end{proof}

For the remainder of this section $M$ has non-empty boundary and we study critical points of optimization problems associated with Examples~\ref{ex:Steklovdensity} and~\ref{ex:boundarypotential}. It turns out that they correspond to free boundary harmonic maps to the unit ball $\mathbb{B}^{k+1}$, similarly to the situation for Steklov eigenvalues on surfaces, see~\cite{FS:extremal, KM}.

\begin{definition}
A smooth map $u\colon (M,g)\to\mathbb{B}^{k+1}$ is called {\em free boundary harmonic} if $u^{-1}(\mathbb{S}^k) = \bd M$, $\Delta_g u = 0$ in the interior of $M$ and $\partial_n u\perp \mathbb{S}^k$.  
\end{definition}
Equivalently, $u$ is a free boundary harmonic map if components of $u$ are eigenfunctions of the problem
\begin{equation*}
\begin{cases}
\Delta_g f = \nu f  &\text{ in $M$};\\
\bd_n f = |\bd_n u| f &\text{ on $\bd M$}\\
\end{cases}
\end{equation*}
with eigenvalue $\nu = 0$. Note that this is exactly the problem from Example~\ref{ex:boundarypotential} with $v=|\bd_n u|$. We say that $\ind_S(u) = m$ if $|\bd_nu|\in \spaceV^{\bd}_{m+1}\setminus \spaceV^{\bd}_m$, i.e. $\nu_m(|\bd_nu|)<\nu_{m+1}(|\bd_nu|) = 0$. Similarly, $u$ is a free boundary harmonic map if components of $u$ are eigenfunctions of the problem
\begin{equation*}
\begin{cases}
\Delta_g f = 0  &\text{ in $M$};\\
\bd_n f = \sigma|\bd_n u| f &\text{ on $\bd M$}\\
\end{cases}
\end{equation*}
with eigenvalue $\sigma=1$. This is exactly the problem from Example~\ref{ex:Steklovdensity} with $\rho=|\bd_n u|$. Then $\ind_S(u) =m$ iff $\sigma_{m-1}(|\bd_n u|)<\sigma_{m}(|\bd_n u|)=1$.

 The following definition is analogous to Definition~\ref{def:1}.

\begin{definition}
Let $0<\rho\in C^\infty(\bd M)$. We say that $\rho$ is a critical density the functional 
\begin{equation}
\label{def:functional3}
G_m(\rho) = \sigma_m(\rho)\int_{\bd M}\rho\,ds_g
\end{equation}
if for any smooth deformation $\rho(t)>0,$ $\rho(0) = \rho$, $t\in(-\eps,\eps)$ one has
\begin{equation*}
G_m(\rho(t))\leq G_m(\rho) + o(t)\qquad\text{or}\qquad G_m(\rho(t))\geq G_m(\rho) + o(t)
\end{equation*}
as $t\to 0$.
\end{definition} 

\begin{proposition}
\label{prop:critical3}
Let $0<\rho\in C^\infty(\bd M)$  be a critical point of the functional~\eqref{def:functional3}.
Then there exists a smooth free boundary harmonic map $u\colon (M,g)\to \mathbb{B}^{k+1}$ and a positive constant  such that $\sigma_m(\rho)\rho = |\bd_n u|$ and the components of $u$ are $\sigma_m(\rho)$-eigenfunctions. In particular, $\ind_S(u)\leq m$. 

Conversely, let $u\colon (M,g)\to\mathbb{B}^{k+1}$ be a smooth free boundary harmonic map. Then $|\partial_nu |$ is a critical point of the functional~\eqref{def:functional3} with $m = \ind_S(u)$.
\end{proposition}
\begin{proof}
The proof is similar to that of Proposition~\ref{prop:critical1}. The computation in the proof of~\cite[Theorem 9]{KM} gives that the quadratic form
$$
B(\rho'(0))[\phi,\psi] = -\sigma_m(\rho)\int_{\bd M}\rho'(0)\phi\psi\,ds_g
$$
on $\sigma_m(\rho)$-eigenspace is sign indefinite as long as $\int_{\bd M}\rho'(0)\,ds_g=0$.
Combining with~\cite[Lemma 1]{KM} 
we deduce the existence of $\sigma_m(\rho)$-eigenfunctions $u_1,\ldots, u_{k+1}$ such that
$\sum_{i=1}^{k+1} (u_i)^2 = 1$ on $\bd M$. Setting $u = (u_1,\ldots, u_{k+1})$ one has $\Delta_g(|u|^2) = -2|du|_g^2\leq 0$, therefore, by the maximum principle $|u|\leq 1$ with equality only on $\bd M$. Furthermore, $\bd_n u = \sigma_m(\rho)\rho u = |\bd_n u|u$, so $u$ is a free boundary harmonic map and $\sigma_m(\rho)\rho = |\bd_n u|$.
\end{proof}

Finally, we state the corresponding result for potentials.

\begin{definition}
Let $v\in C^\infty(\bd M)$. We say that $v$ is a critical potential for the functional 
\begin{equation}
\label{def:functional4}
v\mapsto \int_Mv\,ds_g
\end{equation}
in $\spaceV^\bd_m\cap C^\infty(\bd M)$ if for any smooth deformation $v(t)\in \spaceV^\bd_m\cap C^\infty(\bd M)$, $v(0) = v$, $0\leq t < \eps$ one has 
\begin{equation*}
\int v(t)\,ds_g\leq\int v\,dv_g + o(t)
\end{equation*}
as $t\to 0+$.
\end{definition}

\begin{proposition}
Let $v\in\spaceV^{\bd}_m$ be a smooth potential critical for the functional~\eqref{def:functional4}
 in $\spaceV^{\bd}_m$. Then there exists a smooth free boundary harmonic map $u\colon (M,g)\to \mathbb{B}^{k+1}$ such that $|\bd_n u| = v$ and  $\nu_m(v\,ds_g) = 0$. In particular, $\ind_S(u)\leq m$. 

Conversely, let $u\colon (M,g)\to\mathbb{B}^{k+1}$ be a smooth free boundary harmonic map. Then $|\bd_n|\in\spaceV^\bd_{\ind_S(u)}$ is a critical potential for the functional~\eqref{def:functional4} in $\spaceV^\bd_{\ind_S(u)}$.
\end{proposition}
\begin{proof}
The proof is completely analogous to the proof of Proposition~\eqref{prop:critical2}. The only difference is in the expression for the quadratic form~\eqref{Aform}. For deformations $v(t)$ the eigenvalues $\nu_m(v(t)\,ds_g)$ change according to the eigenvalues of 
$$
B(v'(0))[\phi,\psi] = -\int_{\bd M}v'(0)\phi\psi\,ds_g.
$$
The same argument yields the fact that $B$ is sign indefinite as long as $\int_{\bd M}v'(0)\,ds_g = 0$.
One then concludes the existence of eigenfunctions $u_1,\ldots, u_{k+1}$ such that $\sum_{i=1}^{k+1} (u_i)^2 = 1$ on $\bd M$. The end of the proof is the same as in Proposition~\ref{prop:critical3}.
\end{proof}

\begin{remark}
The last two results suggest that it should be possible to develop existence theory for free boundary harmonic maps along the lines of what is done in Sections 2 and 3. 
\end{remark}

\section{Appendix}

\subsection{Properties of the distance function to a submanifold $N\subset \mathbb{R}^L$}\label{app.1}

As in Section \ref{sec:minmax}, let $N^k$ be a closed manifold of dimension $k$ embedded isometrically $N\subset \mathbb{R}^L$ in the Euclidean space $\mathbb{R}^L$. Let $U=B_{\delta_0}(N)\subset\mathbb{R}^L$ be a tubular neighborhood on which the nearest-point projection
$$\Pi_N\colon U\to N$$
is smooth, and for every $x\in U$, denote by $P(x),P^{\perp}(x)\in \mathrm{End}(\mathbb{R}^L)$ the projections
$$P(x):=P_{T_{\Pi(x)}N}\text{ onto }T_{\Pi(x)}N$$
(viewing $T_{\Pi(x)}N$ as a subspace of $\mathbb{R}^L$) and
$$P^{\perp}(x)=I-P(x).$$
For $x\in N$, denote by $\II_N(x)\in \mathrm{Sym}^2(TN)\otimes T^{\perp}N$ the second fundamental form 
$$\II_N(x)(X,Y):=(D_XY)^{\perp}\text{ for }X,Y\in \Gamma(TN),$$
and for general $x\in U$, denote by $B(x)\in \mathrm{Sym}^2(\mathbb{R}^L)\otimes \mathbb{R}^L$ the tensor
$$B(x)(X,Y):=\II_N(\Pi_N(x))(P(x)X, P(x)Y).$$

Since the nonlinear potential term in the energies $E_{\epsilon}$ defined in Section \ref{sec:minmax} coincides with the smooth function $\frac{1}{2}d_N^2$ in a neighborhood of $N$, it will be useful for us to record the following estimate for $\Hess(d_N^2)$. The following proposition is proved in~\cite[Propositions 3.3, 3.5]{Man}.

\begin{lemma}[Mantegazza~\cite{Man}]
\label{dsquared.lem} There is a constant $C(N)$ such that on the tubular neighborhood $U=B_{\delta_0}(N)$,
\begin{equation*}
\begin{split}
&\left|\Hess\left(\frac{1}{2}d_N^2\right)(x)(X,Y)-\langle X, P^{\perp}(x)Y\rangle+d_N(x)\langle B(x)(X,Y),\nabla d_N(x)\rangle\right|\\
&\leq Cd_N^2(x)|X||Y|
\end{split}
\end{equation*}
for any $x\in U$ and $X,Y\in \mathbb{R}^L$.
\end{lemma}

\subsection{Proofs of monotonicity and small energy regularity for the Ginzburg-Landau approximation}

The energy monotonicity and small energy regularity results for functionals like those defined in Section 2.1 above are well known to experts (see in particular \cite{CS}), but since it is somewhat difficult to find a complete, correct proof for the small energy regularity statement in the literature, we include the arguments below for the reader's convenience.

\begin{lemma}\label{mono}
Let $u\colon M\to \mathbb{R}^L$ be a critical point for $E_{\epsilon}$. Then on geodesic balls $B_r(p)$ of radius $r<\inj(M)$, we have
$$\frac{d}{dr}\left(e^{Cr^2}r^{2-n}\int_{B_r(p)}e_{\epsilon}(u)\right) \geq e^{Cr^2}r^{2-n}\left(\int_{\partial B_r(p)}\left|\frac{\partial u_{\epsilon}}{\partial \nu_p}\right|^2+\frac{2}{r}\int_{B_r(p)}\frac{W(u_{\epsilon})}{\epsilon^2}\right),$$
where $\nu_p$ denotes the gradient of the distance function $d_p$ to $p$, and $C=C(n,k)$ is a constant depending on the geometry of $(M,g)$ only through the dimension $n=\dim M$ and a sectional curvature bound $k\geq |\mathrm{Sec}_M|$.
\end{lemma}
\begin{proof} The map $u\colon M\to \mathbb{R}^L$ is a critical point for $E_{\epsilon}$ if and only if it satisfies \eqref{gl.eqn}, and it is straightforward to check that the two-tensor
$$T_{\epsilon}(u):=e_{\epsilon}(u_{\epsilon})g-du_{\epsilon}^*du_{\epsilon}$$
must be divergence-free for any map $u$ solving \eqref{gl.eqn}, where we've set
$$e_{\epsilon}(u):=\frac{1}{2}|du|^2+\frac{W(u)}{\epsilon^2}.$$

Pairing the condition $\mathrm{div}(T_{\epsilon}(u))=0$ against the vector field $X(p)=\frac{1}{2}\nabla d_p^2$ on a small geodesic ball $B_r(p)$, one finds
$$\int_{B_r(p)}\langle T_{\epsilon}(u),DX\rangle=\int_{\partial B_r(p)}T_{\epsilon}(u)(X,\nu_p)=r\int_{\partial B_r(p)}T_{\epsilon}(u)(\nu_p,\nu_p).$$
It follows from the Hessian comparison theorem for the distance function $d_p$ that $|DX-g|\leq Cd_p^2$ for some constant $C(n,k)$ depending only on $n=\dim M$ and a sectional curvature bound $|\mathrm{Sec}_M|\leq k$, so the preceding identity yields
$$r\int_{\partial B_r(p)}T_{\epsilon}(u)(\nu_p,\nu_p)\geq \int_{B_r(p)}\langle T_{\epsilon}(u),g\rangle-C'(n,k) r^2\int_{B_r(p)}e_{\epsilon}(u).$$
By definition of $T_{\epsilon}(u)$, we can rewrite the above as
\begin{equation*}
\begin{split}
&r\int_{\partial B_r(p)}\left(e_{\epsilon}(u_{\epsilon})-\left|\frac{\partial u_{\epsilon}}{\partial\nu_p}\right|^2\right)\geq \\
\geq&\int_{B_r(p)}(n-2)e_{\epsilon}(u)+2\frac{W(u)}{\epsilon^2}-C'(n,k)r^2\int_{B_r(p)}e_{\epsilon}(u),
\end{split}
\end{equation*}
from which the desired monotonicity statement follows.
\end{proof}

\begin{lemma}\label{gl.epsreg} There exists a constant $\eta_0(N,n,k)>0$ depending only on the target manifold $N\subset \mathbb{R}^L$, the dimension $n=\dim M$, and a sectional curvature bound $|\mathrm{Sec}(M,g)|\leq k$, such that if $u\colon B_r(p)\to N$ solves \eqref{gl.eqn} on a ball $B_{2r}(p)\subset M$ with $2r<\min\{\inj(M,g),1\}$, and
$$r^{2-n}E_{\epsilon}(u;B_{2r}(p))<\eta_0,$$
then $r^2 e_{\epsilon}(u)\leq 1$ on $B_{r/2}(p)$.
\end{lemma}

\begin{proof} Here we follow essentially the same argument as in \cite{CS}, taking care in our computations to avoid some minor errors introduced in that paper. Note first that for solutions of \eqref{gl.eqn}, the Bochner formula gives
\begin{equation*}
\begin{split}
-&d^*de_{\epsilon}(u)=\\
=&|\Hess(u)|^2+\langle \Ric(g),du^*du\rangle-\langle du, d\Delta u\rangle+\mathrm{div}\left(\frac{D_jW(u)}{\epsilon^2}du^j\right)\\
=&|\Hess(u)|^2+\langle \Ric(g),du^*du\rangle
+2\epsilon^{-2}D^2_{ij}W(u)\langle du^i,du^j\rangle+\frac{|DW(u)|^2}{\epsilon^4}.
\end{split}
\end{equation*}
By Lemma \ref{dsquared.lem}, we know that
$$\Hess(d_N^2)\geq -C(N) d_N g_{\mathbb{R}^L}$$
as quadratic forms on a tubular neighborhood of $N\subset \mathbb{R}^L$, so since $W(u)=\frac{1}{2}d_N(u)^2$ where $d_N(u)\leq \delta_0/2$, it follows that
$$D^2W(u)\geq -C(N)|DW(u)|g_{\mathbb{R}^L}\text{ where }d_N(u)\leq \delta_0/2,$$
and it is straightforward to check that
$$|D^2W(u)|\leq C(N)W(u)\text{ where }d_N(u)\geq \delta_0/2,$$
so that
\begin{equation}\label{hess.w.est}
D^2W(u)\geq -C(N)\left(|DW(u)|+W(u)\right)g_{\mathbb{R}^L}
\end{equation}
holds for any value of $u$.

Using \eqref{hess.w.est} and a simple application of Young's inequality, we see that
\begin{equation*}
\begin{split}
&2\epsilon^{-2}D^2_{ij}W(u)\langle du^i,du^j\rangle+\frac{|DW(u)|^2}{\epsilon^4}\geq\\
\geq& -C(N)\epsilon^{-2}\left(|DW(u)|+W(u)\right)|du|^2+\frac{|DW(u)|^2}{\epsilon^4}\\
\geq &-C(N)^2|du|^4-\frac{|DW(u)|^2}{\epsilon^4}-\frac{W(u)^2}{\epsilon^4}+\frac{|DW(u)|^2}{\epsilon^4}\\
\geq & -C(N)^2|du|^4-\frac{W(u)^2}{\epsilon^4}\geq -C'(N)e_{\epsilon}(u)^2.
\end{split}
\end{equation*}
Returning to the Bochner identity computation, we arrive at the following estimate.
\begin{equation}\label{gl.boch}
-d^*d e_{\epsilon}(u)\geq -C(n,k)e_{\epsilon}(u)-C(N)e_{\epsilon}(u)^2.
\end{equation}

Now, define $\psi\in C^1_0(B_r(p))$ by
$$\psi(x):=\dist(x,\partial B_r(p))^2e_{\epsilon}(u)(x),$$
and suppose $\psi$ achieves its max at $x_0\in B_r(p)$. Setting $\sigma_0:=\dist(x_0,\partial B_r(p))/2<1$, observe that
$$\dist(x,\partial B_r(p))\geq \dist(x_0,\partial B_r(p))-\sigma_0\geq \frac{1}{2} \dist(x_0,\partial B_r(p))\text{ for all }x\in B_{\sigma_0}(x_0),$$
and consequently
$$e_{\epsilon}(u)(x)=\frac{\psi(x)}{\dist(x,\partial B_r(p))^2}\leq \frac{4\psi(x)}{\dist(x_0,\partial B_r(p))^2}\leq \frac{4\psi(x_0)}{\dist(x_0,\partial B_r(p))^2}\leq 4e_{\epsilon}(u)(x_0)$$
for all $x\in B_{\sigma_0}(p)$. In particular, on the ball $B_{\sigma_0}(x_0)$, it follows from \eqref{gl.boch} that
\begin{equation}\label{bsigma.boch}
\Delta e_{\epsilon}(u)\leq C(n,k)e_{\epsilon}(u)+C(N)e_{\epsilon}(u)(x_0)e_{\epsilon}(u).
\end{equation}

Next, for any smooth function $f\in C^{\infty}(M)$, standard computations give
\begin{equation*}
\begin{split}
&\frac{d}{ds}\left(s^{1-n}\int_{\partial B_s(x_0)}f\right)=\\
=&s^{1-n}\int_{\partial B_s(x_0)}f(-\Delta d_{x_0})+(1-n)s^{-n}\int_{\partial B_s(x_0)}f-s^{1-n}\int_{B_s(x_0)}\Delta f\\
=&-s^{1-n}\int_{B_s(x_0)}\Delta f+s^{1-n}\int_{\partial B_s(x_0)}\left(\frac{(1-n)}{d_{x_0}}-\Delta d_{x_0}\right),
\end{split}
\end{equation*}
while the Hessian comparison theorem gives
$$\frac{1-n}{d_{x_0}}-\Delta d_{x_0}\geq -C(n,k),$$
so that we arrive at the mean value inequality
\begin{equation}\label{mvi.app}
\frac{d}{ds}\left(s^{1-n}\int_{\partial B_s(x_0)}f\right)\geq-s^{1-n}\int_{B_s(x_0)}\Delta f-C(n,k)s^{1-n}\int_{\partial B_s(x_0)}f.
\end{equation}
In particular, taking $f=e_{\epsilon}(u)$ and applying \eqref{bsigma.boch} gives
\begin{equation*}
\begin{split}
&\frac{d}{ds}\left(s^{1-n}\int_{\partial B_s(x_0)}e_{\epsilon}(u)\right)\geq \\
\geq&-C(N,n,k)(1+e_{\epsilon}(u)(x_0))s^{1-n}\int_{B_s(x_0)}e_{\epsilon}(u)-C(n,k)s^{1-n}\int_{\partial B_s(x_0)}e_{\epsilon}(u)
\end{split}
\end{equation*}
for all $s\in [0,\sigma_0)$. 

For a suitable constant $C=C(N,n,k)$, it follows that
$$\frac{d}{ds}\left(s^{1-n}e^{Cs}\int_{\partial B_s(x_0)}e_{\epsilon}(u)\right)\geq -C [1+e_{\epsilon}(u)(x_0)]s^{1-n}\int_{B_s(x_0)}e_{\epsilon}(u)$$
for all $s\in [0,\sigma_0)$. Note, moreover, that for $s\in [0,\sigma_0)$ we have
\begin{equation*}
\begin{split}
&s^{1-n}\int_{B_s(x_0)}e_{\epsilon}(u)\leq\left(s^{-n}\int_{B_s(x_0)}e_{\epsilon}(u)\right)^{1/2}\left(s^{2-n}\int_{B_s(x_0)}e_{\epsilon}(u)\right)^{1/2}\leq\\
&\text{(since $e_{\epsilon}(u)\leq 4e_{\epsilon}(u)(x_0))$ }\leq 2(e_{\epsilon}(u)(x_0))^{1/2}\left(s^{2-n}\int_{B_s(x_0)}e_{\epsilon}(u)\right)^{1/2}
\end{split}
\end{equation*}
At the same time, by Lemma \ref{mono}, we have
$$s^{2-n}\int_{B_s(x_0)}e_{\epsilon}(u)\leq C'(n,k)r^{2-n}\int_{B_r(x_0)}e_{\epsilon}(u)\leq C'(n,k)r^{2-n}E_{\epsilon}(u;B_{2r}(p_0)).$$
Thus, assuming that
\begin{equation}
r^{2-n}E_{\epsilon}(u;B_{2r}(p))<\eta,
\end{equation}
we can combine the estimates above to arrive at an inequality of the form
$$\frac{d}{ds}\left(s^{1-n}e^{Cs}\int_{\partial B_s(x_0)}e_{\epsilon}(u)\right)\geq -C(N,n,k)[1+e_{\epsilon}(u)(x_0)]e_{\epsilon}(u)(x_0)^{1/2}\eta^{1/2}$$
for all $s\in [0,\sigma_0)$. 

In particular, for any $\sigma\in [0,\sigma_0)$, upon integrating the preceding inequality over $s\in [0,\sigma)$, we see that
\begin{eqnarray*}
e_{\epsilon}(u)(x_0)&\leq & C\sigma^{-n}E_{\epsilon}(u;B_{\sigma}(x_0))+C\sigma\eta^{1/2}[1+e_{\epsilon}(u)(x_0)]\sqrt{e_{\epsilon}(u)(x_0)}\\
&\leq &C_0(N,n,k)\cdot \left(\sigma^{-2}\eta+\sigma\eta^{1/2}[1+e_{\epsilon}(u)(x_0)]\sqrt{e_{\epsilon}(u)(x_0)}\right).
\end{eqnarray*}
Multiplying through by $\sigma^2$ and setting $\beta(\sigma):=\sigma^2e_{\epsilon}(u)(x_0)$, we can rewrite this estimate as
\begin{equation}\label{beta.sig.ineq}
\beta(\sigma)\leq C_0\cdot \left(\eta+\eta^{1/2}[\sigma^2+\beta(\sigma)]\sqrt{\beta(\sigma)}\right)
\end{equation}
Recall that 
$$\beta(\sigma_0)=\sigma_0^2e_{\epsilon}(u)(x_0)=\frac{1}{4}\max_{x\in B_r(p)}\dist(x,\partial B_r(p))^2e_{\epsilon}(u)(x),$$
so to complete the proof of the theorem, it suffices to show that
$$\beta(\sigma_0)=\sigma_0^2e_{\epsilon}(u)(x_0)\leq 1/16$$
provided $\eta$ is sufficiently small. Indeed, if we assume that $\beta(\sigma_0)>1/16$, then there must be some $\sigma\in (0,\sigma_0)$ for which
$$\beta(\sigma)=\sigma^2e_{\epsilon}(u)(x_0)=1/16.$$
At this $\sigma\in (0,\sigma_0)\subset (0,1)$, the inequality \eqref{beta.sig.ineq} gives
\begin{equation}
\frac{1}{16}\leq C_0\cdot \left(\eta+\eta^{1/2}\left(1+\frac{1}{16}\right)\frac{1}{4}\right),
\end{equation}
but clearly this cannot hold for $\eta<\eta_0(N,n,k)$ given by $C_0\left(\eta_0+\eta_0^{1/2}\right)=\frac{1}{100}$. This completes the proof.

\end{proof}

\begin{remark}\label{gen.rk} Examining the proof of Lemma \ref{gl.epsreg}, one sees that just two properties of $e_{\epsilon}(u)$ play an essential role: the elliptic inequality \eqref{gl.boch} and the energy monotonicity stated in Lemma \ref{mono}. Using the same proof, one can likewise argue that for any nonnegative function $0\leq f\in C^{\infty}(M)$ satisfying 
$$\frac{d}{dr}\left(e^{A_1r^2}r^{2-n}\int_{B_r(p)}f\right)\geq 0$$
and
$$d^*d f\leq A_2(f+f^2)$$
on $M$ for some constants $A_1,A_2$, there exists $\eta_0(M,A_1,A_2)>0$ such that if 
$$r^{2-n}\int_{B_{2r}(p)}f<\eta_0$$
for some $2r<\min\{\inj(M,g),1\}$, then
$$\|f\|_{L^{\infty}(B_{r/2}(p))}\leq 1/r^2.$$

\end{remark}

\subsection{Universal lower bound on the energy of sphere-valued harmonic maps}

In this subsection we prove the following proposition, which is needed for the proof of Lemma \ref{no.ind.conc}.

\begin{proposition}\label{skener.prop} For any closed Riemannian manifold $(M^n,g)$, there exists a positive constant $\beta(M)>0$ such that 
$$E(u)\geq \beta$$
for any nonconstant harmonic map $u\colon M\to \mathbb{S}^k$ to the unit sphere of any dimension $k\in \mathbb{N}$.

\end{proposition}

The main ingredient needed to prove the proposition is the following simple lemma, observing that the constants in a simplified version of the small energy regularity theorem for sphere-valued harmonic maps--originally proved in \cite{Evans}--do not depend on the dimension $k$ of the target sphere $\mathbb{S}^k$.

\begin{lemma}\label{uni.eps} On a closed Riemannian manifold $(M^n,g)$, there exist constants $\beta_0(M)>0$ and $C(M)<\infty$ such that if $u\colon M\to \mathbb{S}^k$ is a harmonic map to the unit sphere of any dimension $k\in \mathbb{N}$ of energy
$$E(u)<\beta_0,$$
then 
$$\|du\|_{L^{\infty}(M)}^2<C E(u).$$

\end{lemma}

\begin{proof}
Recall that for sphere-valued harmonic maps $u\colon M\to \mathbb{S}^k$, the Bochner identity gives
$$-\frac{1}{2}d^*d|du|^2=|\Hess(u)|^2+\langle \mathrm{Ric}_M,du^*du\rangle-|du|^4,$$
so in particular we have the inequality
\begin{equation}\label{du.subeq}
d^*d|du|^2\leq C_M(|du|^2+|du|^4),
\end{equation}
where the dimension $k$ of the target sphere $\mathbb{S}^k$ plays no role. Likewise, the standard monotonicity identity for harmonic maps gives
\begin{equation}\label{hm.mono}
\frac{d}{dr}\left(e^{C_Mr^2}r^{2-n}\int_{B_r(p)}|du|^2\right)\geq 0,
\end{equation}
where, as in Lemma \ref{mono}, the constant $C=C_M$ depends only on the dimension $n$ and curvature bounds of the domain manifold $(M^n,g)$. 

In particular, per Remark \ref{gen.rk}, setting $\delta_M=\inj(M)/2$, it follows that there exists some $\eta_0(M)<0$ independent of $k$ such that if
\begin{equation}\label{small.inj.en}
\delta^{2-n}\int_{B_{2\delta}(p)}|du|^2<\eta_0,
\end{equation}
then
$$\|du\|_{L^{\infty}(B_{\delta/2}(p))}\leq 1/\delta.$$
Thus, if
$$E(u)<\beta_0,$$
where $\beta_0:=\frac{1}{2}\delta_M^{n-2}\eta_0,$ then we can apply the preceding estimate on the ball $B_{2\delta}(p)$ of radius $2\delta_M=\inj(M)$ at every point $p\in M$, obtaining
$$\|du\|_{L^{\infty}(M)}\leq C(M)=1/\delta_M.$$

Returning to the Bochner formula for $|du|$ and estimating $|du|^4\leq C_M |du|^2$, we deduce that if $E(u)<\beta_0$, then 
$$dd^*|du|^2\leq C_M'|du|^2.$$
From here, we may apply the mean value inequality \eqref{mvi.app} with $f=|du|^2$ to see that for any $x_0\in M$ and $s\in (0,\inj(M))$,
$$\frac{d}{ds}\left(s^{1-n}\int_{\partial B_s(x_0)}|du|^2\right)\geq -C_M's^{1-n}\int_{B_s(x_0)}|du|^2-C_Ms^{1-n}\int_{\partial B_s(x_0)}|du|^2,$$
whence
$$\frac{d}{ds}\left(s^{1-n}e^{Cs}\int_{\partial B_s(x_0)}|du|^2\right)\geq -C s^{1-n}\int_{B_s(x_0)}|du|^2$$
for a suitable constant $C=C(M)$ still independent of $k$. In particular, taking $x_0$ such that $|du|^2(x_0)=\max_{p\in M}|du|^2(p)$, a simple application of H\"older's inequality to the preceding estimate gives
$$\frac{d}{ds}\left(s^{1-n}e^{Cs}\int_{\partial B_s(x_0)}|du|^2\right)\geq -C' \left(s^{2-n}\int_{B_s(x_0)}|du|^2\right)^{1/2}|du|(x_0),$$
while \eqref{hm.mono} gives
$$s^{2-n}\int_{B_s(x_0)}|du|^2\leq C E(u),$$
so that
$$\frac{d}{ds}\left(s^{1-n}e^{Cs}\int_{\partial B_s(x_0)}|du|^2\right)\geq -C \sqrt{E(u)}|du|(x_0).$$
Choosing $t\in [\delta_M/2,\delta_M]$ such that
$$\int_{\partial B_t(x_0)}|du|^2\leq \frac{2}{\delta_M}\int_{B_{\delta_M}(x_0)}|du|^2<C'(M)E(u)$$
and integrating the preceding inequality over $s\in [0,t)$, we find that
$$t^{1-n}e^{Ct}\int_{\bd B_t(x_0)}|du|^2-\vol(\mathbb{S}^{n-1})|du|(x_0)^2\geq -Ct \sqrt{E(u)}|du|(x_0).$$
Rearranging and recalling that $t\in [\delta_M/2,\delta_M]$, we obtain an estimate of the form
$$|du|(x_0)^2\leq C_M\left(\sqrt{E(u)}|du(x_0)|+E(u)\right),$$
and by a simple application of Cauchy-Schwarz to the term $\sqrt{E(u)}|du|(x_0)$, we arrive at an inequality of the desired form
$$|du|^2(x_0)\leq CE(u),$$
completing the proof. 
\end{proof}

We can now complete the proof of Proposition \ref{skener.prop}.

\begin{proof}[Proof of Proposition \ref{skener.prop}]

Let $u\colon M^n\to \mathbb{S}^k$ be a harmonic map of small energy
$$E(u)<\beta_0,$$
where $\beta_0$ is the constant from Lemma \ref{uni.eps}, so that
$$\|du\|_{L^{\infty}(M)}^2<C E(u).$$
Fixing any $x_0\in M$, it then follows that
$$|u(x)-u(x_0)|\leq C\sqrt{E(u)} \dist(x,x_0)$$
for any $x\in M$. In particular if 
$$E(u)<\beta(M):=\min\{\beta_0,(C\diam(M))^{-2}\},$$
writing $e_0:=u(x_0)$, it follows that
$$\langle u(x),e_0\rangle>\frac{1}{2}$$
for all $x\in M$. But it is easy to see that this forces $u$ to be constant: if $\langle u(x),e_0\rangle>\frac{1}{2}$, intregrating the harmonic map equation
$$\Delta u=|du|^2u$$
against the constant vector $e_0$ gives
$$0=\int |du|^2\langle u,e_0\rangle\geq\frac{1}{2}\int_M|du|^2.$$
This completes the proof.

\end{proof}

\end{document}